\documentclass[a4paper,reqno]{amsart}

\usepackage{mathtools}
\mathtoolsset{showonlyrefs}

\usepackage{latexsym}
\usepackage[english]{babel}
\usepackage{fancyhdr}
\usepackage[mathscr]{eucal}
\usepackage{amsmath}
\usepackage{mathrsfs}
\usepackage{mathtools}
\usepackage{amsthm}
\usepackage{amsfonts}
\usepackage{amssymb}
\usepackage{amscd}
\usepackage{geometry}
\usepackage{bbm}
\usepackage{graphicx}
\usepackage{graphicx}
\usepackage{graphics}
\usepackage{latexsym}
\usepackage{color}
\usepackage{pifont}
\usepackage{tikz}
\usepackage[normalem]{ulem}
\usepackage{changepage}

\usepackage{geometry}

\newcommand{\ii}{\mathrm{i}}

\newcommand{\abs}[1]{\left\lvert#1\right\rvert}

\newcommand{\C}{\mathbb C}

\newcommand{\N}{\mathbb N}
\newcommand{\R}{\mathbb R}

\newcommand{\ip}[2]{\langle #1,#2\rangle}
\newcommand{\norm}[1]{\left\|#1\right\|}

\newcommand{\Ncal}{\mathcal{N}}

\theoremstyle{plain}
\newtheorem{theorem}{Theorem}[section]
\newtheorem{lemma}[theorem]{Lemma}
\newtheorem{corollary}[theorem]{Corollary}
\newtheorem{proposition}[theorem]{Proposition}

\theoremstyle{definition}

\newtheorem{remark}[theorem]{Remark}
\newtheorem*{remark*}{Remark}

\numberwithin{equation}{section}

\begin{document}

\title[Scattering for Choquard equations]
{Scattering for Short-Range Mass-Subcritical Choquard Equations}

 \author[A.~Michelangeli]{Alessandro Michelangeli}
 \address[A.~Michelangeli]{Mathematics and Science Department, AUBG Blagoevgrad \\
 and Hausdorff Center for Mathematics, HCM Bonn\\ and  Trieste Institute for Theoretical Quantum Technologies, TQT Trieste}
 \email{amichelangeli@aubg.edu}
 
  \author[M.~Tarulli]{Mirko Tarulli}
 \address[M.~Tarulli]{Mathematics and Science Department, AUBG Blagoevgrad \\ and
 Institute of Mathematics and Informatics, Bulgarian Academy of Science, Sofia}
 \email{mtarulli@aubg.edu}

% \date{\today}
%
% \subjclass[2010]{...}
%
% \keywords{....}

% \thanks{\emph{Acknowledgements.} For this work, A.M.~gratefully acknowledges support from the David Flanagan funds at the AUBG, American University Bulgaria, and the Alexander von Humboldt Foundation, Germany. }

 \begin{abstract}
We study scattering for the defocusing inhomogeneous generalized
Choquard equation with data in the conformal energy space. In the full
short-range mass-subcritical regime, a key
estimate for the nonlinear term yields scattering in \(L^2\). A 
pseudoconformal argument  in the compactified variables then gives scattering in \(H^1\). In the
homogeneous case, estimates based on the Galilean vector field further
imply scattering in \(\Sigma\) above the Strauss threshold. The method
avoids the additional algebraic restrictions arising from intermediate
Strichartz-Sobolev estimates.
\end{abstract}

\maketitle

% \tableofcontents
%
%  \newpage

\section{Introduction and main results}
The main object of this paper is to study the long-time properties of global solutions to
the defocusing inhomogeneous generalized Choquard equation.  More precisely, we consider the Cauchy problem
\begin{equation}\label{eq:ME}
 \left\{
 \begin{aligned}
  \ii\partial_t\phi+\Delta\phi
  &=
  \mathcal N_{b,\mu}(\phi),
  \qquad (t,x)\in\R\times\R^d,
  \\
  \phi(0,x)
  &=
  \phi_0(x),
  \qquad x\in\R^d,
 \end{aligned}
 \right.
\end{equation}
where
\(
 \phi:\R\times\R^d\longrightarrow\C.
\)
The nonlinearity is given by
\begin{equation}\label{eq:Nonl}
 \mathcal N_{b,\mu}(f)(x)
 :=
 \Bigl(
  I_\mu*(|\cdot|^{-b}|f|^q)
 \Bigr)(x)
 |x|^{-b}|f(x)|^{q-2}f(x),
 \qquad
 I_\mu(x):=|x|^{-(d-\mu)}.
\end{equation}
Unless otherwise specified, the parameters satisfy
\(
 d\ge1,
 \,
 0<\mu<d,
 \,
 b\ge0,
 \,
 q\ge2.
\)
The Cauchy problem is posed in the energy space \(H^1(\R^d)\).  We require initial data \(
 \phi_0\in \Sigma,
\)
where
\begin{equation}\label{eq:introS}
 \Sigma
 :=
 \bigl\{
  f\in H^1(\R^d):xf\in L^2(\R^d)
 \bigr\},
\end{equation}
is the conformal space equipped with the norm
\[
 \|f\|_\Sigma^2
 :=
 \|f\|_{H^1}^2+\|xf\|_{L^2}^2.
\]
We recall that the equation in \eqref{eq:ME} is invariant under the scaling
\[
 \phi_\lambda(t,x)
 :=
 \lambda^{\frac{2+\mu-2b}{2(q-1)}}
 \phi(\lambda^2t,\lambda x),
 \qquad
 \lambda>0.
\]
The corresponding Sobolev critical exponent is then given by
\[
 s_c
 :=
 \frac d2
 -
 \frac{2+\mu-2b}{2(q-1)},
\]
and the mass critical exponent is 
\begin{equation}\label{eq:BMa}
 q_{\mathrm m}
 :=
 1+\frac{2+\mu-2b}{d}.
\end{equation}
When \(d\ge3\), the energy-critical reads as
\[
 q_{\mathrm e}
 :=
 1+\frac{2+\mu-2b}{d-2},
\]
whereas for \(d=1,2\) we choose \(q_{\mathrm e}:=+\infty\). Our analysis concerns the short-range mass-subcritical interval
%\begin{equation}%\label{eq:ISRS}
\(
 q_{\mathrm{sr}}
<
 q
 <
 q_{\mathrm m},
 \)
%\end{equation}
with \(q_{\mathrm m}\) defined in \eqref{eq:BMa} and where the lower endpoint
\begin{equation}\label{eq:mass1}
 q_{\mathrm{sr}}
 :=
 1+\frac{1+\mu-2b}{d}
\end{equation}
is the short-range threshold. We first clarify the Cauchy theory used in the statements below.  For
the inhomogeneous equation \(b>0\), we assume
\begin{equation}\label{eq:IWC}
 0<2b<
 \min\left\{
  2d,\,
  2+\mu,\,
  d+\mu,\,
  4-d+\mu
 \right\}.
\end{equation}
Under the assumption \eqref{eq:IWC}, the local energy theory for
the inhomogeneous generalized Choquard equation is available for
\(
 2\le q<q_{\mathrm e},
\)
see
\cite{AlharbiSaanouni2019,AroraRoudenko2022,BoulaarasSaanouni2026}.
Consider the potential energy described by the functional
\begin{equation}\label{eq:weightedE}
 \mathcal I_{b,\mu}(f)
 :=
 \iint_{\R^d\times\R^d}
 \frac{|f(x)|^q|f(y)|^q}
 {|x|^b|x-y|^{d-\mu}|y|^b}\,dx\,dy .
\end{equation}
Notice that for \(b=0\), this reduces to the homogeneous Choquard potential energy. 
Then, the solution of \eqref{eq:ME} satisfies the conservation of mass
and energy:
\begin{align}
 M(\phi(t))
 &:=
 \|\phi(t)\|_{L^2}^2
 =
 \|\phi_0\|_{L^2}^2
 =
 M(\phi_0),
 \label{eq:CM2}
 \\
 E(\phi(t))
 &:=
 \frac12\|\nabla\phi(t)\|_{L^2}^2
 +
 \frac1{2q}\mathcal I_{b,\mu}(\phi(t))
 =
 E(\phi_0),
 \label{eq:CE2}
\end{align}
for every \(t\) in the maximal interval of existence of the solution. In particular, since
the interaction potential is nonnegative, the following bound is satisfied 
\begin{equation} \label{eq:CE3}
 \|\nabla\phi(t)\|_{L^2}^2
 \le
 2E(\phi_0).
\end{equation}
For \(\phi_0\in\Sigma\), the corresponding solution remains in
\(\Sigma\) on every compact subinterval of its maximal interval of
existence and satisfies
\begin{equation}\label{eq:IMP}
 \|x\phi(t)\|_{L^2}
 \le
 \|x\phi_0\|_{L^2}
 +
 2\int_{\min\{0,t\}}^{\max\{0,t\}}
 \|\nabla\phi(\tau)\|_{L^2}\,d\tau.
\end{equation}
In the defocusing case, the above conservation of mass \eqref{eq:CM2} and energy \eqref{eq:CE2}  gives instead
\[
 \sup_{t\in\R}\|\phi(t)\|_{H^1}<\infty.
\]
As a consequence, the solution extends globally in \(H^1\), while \eqref{eq:IMP} ensures that the second
spatial moment remains finite on every bounded time interval, while
Proposition~\ref{prop:FMP} gives
\(
 x\phi\in C(\R;L^2),
\)
and thus 
\(
 \phi\in C(\R;\Sigma).
\)
In the homogeneous case \(b=0\), we use the standard \(H^1\) Cauchy
theory and the same argument used previously, together with Proposition~\ref{prop:FMP},
yields the corresponding global solution in \(C(\R;\Sigma)\). An important feature of the mechanism is that although \eqref{eq:IWC} is invoked in the well-posedness theory, it
imposes no additional restriction in the present mass-subcritical
range. Once the global \(H^1\)-solution is available, the scattering
does not require further algebraic restrictions.

The scattering theory for Schr\"odinger equations with nonlocal
nonlinearities originates in
\cite{GinibreVelo1979,GinibreVelo1980}. The particular case
\(b=0\) and \(q=2\) is the classical Hartree equation
\begin{equation}
i\partial_t\phi+\Delta\phi
=
\bigl(|\cdot|^{-\gamma}*|\phi|^2\bigr)\phi,
\qquad
\gamma:=d-\mu.
\end{equation}
For the defocusing problem, asymptotic completeness in weighted
Sobolev spaces was established in \cite{HayashiTsutsumi1987} in the
range
\(
4/3<\gamma<\min\{4,d\},
\)
together with the existence of \(L^2\)-scattering states for data in
\(\Sigma\) when \(1<\gamma<4/3\). Wave operators for short-range
nonlocal nonlinearities were studied in \cite{NawaOzawa1992}, while
scattering in the energy space for a class of Hartree equations was
established in \cite{Nakanishi1999}. At the threshold \(\gamma=1\),
the problem enters the long-range regime and the standard free
scattering asymptotics must be modified; see
\cite{GinibreOzawa1993}. The theory was subsequently developed at the
critical thresholds. For the defocusing energy-critical Hartree
equation, global well-posedness and scattering for radial data were
proved in \cite{MiaoXuZhao2007}. The same authors treated the
mass-critical Hartree equation with radial data in
\cite{MiaoXuZhao2009}. These results provide important counterparts
to the mass-subcritical problem considered here, where the main
difficulty is instead the long-time integrability of the nonlinearity
in the short-range regime.

For generalized and inhomogeneous Choquard equations, the global
dynamics has been considered in several parameter regimes. Threshold
criteria separating global existence from finite-time blow-up for a
class of inhomogeneous Choquard equations were obtained in
\cite{AlharbiSaanouni2019}. For the focusing generalized Hartree
equation, \cite{AroraRoudenko2022} investigated the global
\(H^1\)-dynamics in the intercritical regime, including scattering
and blow-up behavior relative to the associated variational
threshold. In the defocusing mass-subcritical setting,
\cite{BoulaarasSaanouni2026} obtained nonradial scattering in weighted
Sobolev spaces above a Strauss-type threshold, under additional
restrictions on the parameters. Among the available results for
inhomogeneous generalized Hartree equations, the latter is the
closest to the regime addressed in the present work.

A complementary direction concerns the stationary Choquard problem,
which is relevant to the analysis of standing waves and their
stability. Existence and uniqueness results for ground states of the
\(p\)-Choquard model were proved in
\cite{GeorgievTarulliVenkov2019}. More recently, local uniqueness of
ground states for the generalized Choquard equation in the
energy-subcritical regime was shown in
\cite{GeorgievTarulliVenkov2024}.

The short-range mass-subcritical scattering theory for local nonlinear
Schr\"odinger equations has been extensively developed. Classical
\(L^2\)-scattering results for initial data in \(\Sigma\) were
obtained in \cite{TsutsumiYajima1984}, while the asymptotic behavior
of rapidly decaying solutions was studied in
\cite{CazenaveWeissler1992}. In the inhomogeneous local framework,
scattering in weighted \(L^2\) spaces was studied in
\cite{Dinh2021}. More recently, \(H^1\)-scattering for
mass-subcritical nonlinear Schr\"odinger equations with short-range
nonlinearity and initial data in \(\Sigma\) throughout the
corresponding short-range interval was proved in
\cite{BurqGeorgievTzvetkovVisciglia}. Their argument combines the
pseudoconformal transformation with the asymptotic behavior of the
second spatial moment and the Hilbert space structure of \(H^1\).
Related asymptotic properties of higher spatial moments were analyzed
in \cite{TzvetkovVisciglia2021}.

The contributions of this paper are threefold.  First, for the inhomogeneous equation, we establish an estimate that
controls the nonlinear term directly in terms of the weighted nonlocal
potential functional \(\mathcal I_{b,\mu}\).  This avoids the intermediate
Lebesgue-Sobolev estimates that produce additional exponent
restrictions in a standard Strichartz argument and yields
\(L^2\)-scattering throughout the full short-range mass-subcritical
interval
\(
 q_{\mathrm{sr}}
 <
 q
 <
  q_{\mathrm{m}}.
\)
Second, we adapt the pseudoconformal argument of
\cite{BurqGeorgievTzvetkovVisciglia} to the weighted nonlocal Choquard
interaction. The proof uses the decay properties of the rescaled second moment under
the pseudoconformal scaling to recover the corresponding \(H^1\)-asymptotics.
This strengthens the \(L^2\)-scattering result to \(H^1\)-scattering
throughout the same parameter range.

Third, in the homogeneous case \(b=0\), we prove the stronger
conclusion of scattering in \(\Sigma\) above the Strauss exponent.
The proof proceeds directly from the evolution equation and the
associated weighted estimates and combines the potential
energy with the vector field \(J(t)=x+2it\nabla\).  Unlike the argument based on the combination of decay,
Strichartz estimates, and Sobolev embedding, which retains an additional
algebraic restriction even after choosing \(b=0\),
the present approach closes under the exact condition \(q>q_0\).

To the best of our knowledge, this provides a new scattering construction
for a weighted nonlocal Schr\"odinger equation in which the same
asymptotic state is successively shown to belong to \(L^2\), \(H^1\),
and, in the homogeneous case, \(\Sigma\).

\subsection*{Main results}

Our first main result concerns the inhomogeneous equation and gives
\(H^1\)-scattering throughout the full short-range mass-subcritical
interval.

\begin{theorem}
\label{thm:MH1}
Let
\begin{equation}\label{eq:ass}
d\ge1,
 \qquad
 0<\mu<d,
 \qquad
 0\le b<d,
 \qquad
 2b<d+\mu,
 \qquad
 q\ge2,
\end{equation}
and assume
\begin{equation}\label{eq:frf}
 q_{\mathrm sr}
 <
 q
 <
 q_{\mathrm m}.
\end{equation}
Let \(\phi_0\in\Sigma\), and let
\(
\phi\in C\bigl(\mathbb R;H^1(\mathbb R^d)\bigr)
\)
be the corresponding global solution of
\eqref{eq:ME}. Then, there exist \(\phi_\pm\in H^1(\R^d)\) such that
\begin{equation}
 \lim_{t\to\pm\infty}
 \left\|
 \phi(t)-e^{it\Delta}\phi_\pm
 \right\|_{H^1}
 =0.
\end{equation}
\end{theorem}
We now turn to the homogeneous generalized Choquard equation
\begin{equation}\label{eq:introCb0}
 i\partial_t\phi+\Delta\phi
 =
 \bigl(I_\mu*|\phi|^q\bigr)|\phi|^{q-2}\phi.
\end{equation}
We introduce
\begin{equation}\label{eq:IBN}
 \nu:=d(q-1)-\mu,
\end{equation}
and consider the quadratic polynomial
\begin{equation}\label{eq:IBSP}
 P_0(X)
 :=
 2dX^2+(d-2\mu-2)X-\mu-2.
\end{equation}
One can see \eqref{eq:IBSP} that 
has a unique positive root \(X_0\). The associated Strauss threshold is
defined by
\(
 q_0:=1+X_0,
\)
and a direct calculation gives
\[
 P_0\left(\frac{\mu+1}{d}\right)<0,
 \qquad
 P_0\left(\frac{\mu+2}{d}\right)>0,
\]
that is
\begin{equation}
 1+\frac{\mu+1}{d}
 <
 q_0
 <
 1+\frac{\mu+2}{d}.
\end{equation}

Our second main result gives scattering in the conformal space for the
homogeneous equation.

\begin{theorem}
\label{thm:IMSB}
Let
\begin{equation}\label{eq:ass2}
 d\ge1,
 \qquad
 0<\mu<d,
 \qquad
 q\ge2,
\end{equation}
and assume
\begin{equation}\label{eq:frf2}
 q_0<q<1+\frac{\mu+2}{d}.
\end{equation}
Let \(\phi_0\in\Sigma\), and let \(\phi\in C\bigl(\mathbb R;H^1(\mathbb R^d)\bigr)\) the unique global solution of the Cauchy problem associated to \eqref{eq:introCb0} with \(\phi(0)=\phi_{0}\). Then, there exists \(\phi_{\pm}\in\Sigma\) such that
\begin{equation}\label{eq:IMSC}
 \lim_{t\to\pm\infty}
 \left\|
 e^{-it\Delta}\phi(t)-\phi_{\pm}
 \right\|_{\Sigma}
 =0.
\end{equation}
\end{theorem}

The homogeneous result has a sharper parameter range than that obtained
by formally applying to \(b=0\) the decay, Strichartz, and Sobolev
estimates used in \cite{BoulaarasSaanouni2026}.
Indeed, their intermediate pointwise estimate becomes
\[
 \|\phi(t)\|_{L^r}
 \lesssim
 (1+t)^{-\nu^2/(4q)},
 \qquad
 r=\frac{2dq}{d+\mu},
\]
and the corresponding time-integrability condition is
\begin{equation}\label{eq:ISBC}
 \nu^2(q-1)+\nu-2q>0.
\end{equation}
Thus an additional algebraic restriction persists even after selecting \(b=0\). The potential interaction estimate used here instead yields the
condition
\(
 \nu(2q-1)>2q,
\)
which is precisely \(q>q_0\).  At the algebraic level, the condition \eqref{eq:ISBC} is more restrictive than
\(q>q_0\): if \(X=q_0-1\), then at \(q=q_0\),
\[
 \nu^2(q_0-1)+\nu-2q_0
 =
 -\frac{4X^2(X+1)}{(2X+1)^2}<0.
\]
Consequently, \eqref{eq:ISBC} fails for \(q>q_0\) sufficiently
close to \(q_0\), whereas Theorem~\ref{thm:IMSB} covers the entire range
\(
 q\ge2,
 \)
 \(
 q_0<q<q_{\mathrm{m}}.
\)
 
  It is natural also to ask whether the Lens Transform argument of
\cite{BurqGeorgievTzvetkovVisciglia} can be applied directly to
\eqref{eq:introCb0}.  The transform itself remains compatible with the
homogeneity of the Riesz kernel and yields a corresponding conformal
energy.  The obstruction lies instead in the exponent numerology of the
nonlinear estimates.  Standard Hardy-Littlewood-Sobolev and Sobolev
bounds provide the superlinear continuity inequality required by the
local endpoint argument, but introduce additional spatial-exponent
restrictions and lead to conditions stronger than \(q>q_0\).
Conversely, in the homogeneous case, the structure of
the nonlinearity allows us to combine the potential energy directly
with the vector field
\(
 J(t)=x+2it\nabla.
\)
The resulting estimate for \(J(t)\mathcal N_{0,\mu}(\phi)\) preserves a
power of the potential energy and retains the sharp Strauss
numerology.  Combined with the pseudoconformal decay, it closes under
the exact strict condition
\(
 \nu(2q-1)>2q,
\)
without the additional restrictions arising
from intermediate Sobolev estimates. This explains why the argument based on the use of the operator 
 \(J(t)\) is well suited to the parameter range considered here.

\subsection*{Outline of the paper}

In Section~\ref{sec:Pseudo2}, we introduce the pseudoconformal
transformation of \eqref{eq:ME}, derive the associated conformal-energy
bounds, and establish the weighted interaction estimates needed in the
scattering analysis. Section~\ref{sec:IEB} uses these
estimates to prove strong \(L^2\)-convergence of the compactified
solution and hence \(L^2\)-scattering in the full short-range
mass-subcritical regime. Section~\ref{sec:BPUC} is devoted to the passage from
\(L^2\)-scattering to \(H^1\)-scattering. We show that the spatial gradient of the solution is asymptotically determined by a rescaled \(x\phi\) term and exploit the resulting relation with the rescaled second moment to prove convergence of the gradient norm. The appendices collect the second moment estimate and the auxiliary
inequalities needed in the preceding sections.

\section{Pseudoconformal transformation and bilinear estimates}
\label{sec:Pseudo2}

In this section we study the inhomogeneous Choquard equation in
pseudoconformal variables and base the scattering argument on an
estimate for the nonlocal potential functional.
The homogeneous case \(b=0\) is recovered throughout as a particular
case. In the remainder of the paper, we consider only the limit
\(t\to+\infty\). The Cauchy problem \eqref{eq:ME} is invariant under the transformation
\(
 \tilde \phi(t,x):=\overline{\phi(-t,x)},
\)
thus the corresponding assertions as \(t\to-\infty\) follow by applying the
forward-time argument to the past-time solution and no
generality is lost by restricting the analysis below to nonnegative
times. We introduce the pseudoconformal variables
\begin{equation}\label{eq:inhCon}
 \tau=\frac{t}{1+t},
 \qquad
 y=\frac{x}{1+t},
 \qquad
 \psi(\tau,y)
 =
 (1-\tau)^{-d/2}
 e^{-\frac{i|y|^2}{4(1-\tau)}}
 \phi\left(\frac{\tau}{1-\tau},\frac{y}{1-\tau}\right).
\end{equation}

 \begin{proposition}
\label{prop:BPCP1}
 Assume \eqref{eq:ass} and \eqref{eq:frf}. Let \(\phi_0\in\Sigma\),  and let
\(
\phi\in C\bigl(\mathbb R;H^1(\mathbb R^d)\bigr)
\)
be the corresponding global solution of 
\eqref{eq:ME}. Let \(\psi\) be defined by
\eqref{eq:inhCon}. Then
\(
 \psi(0,x)
 =
 e^{-\frac{i|x|^2}{4}}\phi_0(x)
 \in\Sigma,
\)
\(
 \psi\in C([0,1);\Sigma),
\)
and \(\psi\) solves
\begin{equation}\label{eq:pcpsieq}
 i\partial_\tau\psi+\Delta_y\psi
 =
 (1-\tau)^{-\beta}\mathcal N_{b,\mu}(\psi),
\end{equation}
with \(\mathcal N_{b,\mu}\) defined in
\eqref{eq:Nonl} and where
\begin{equation}\label{eq:beta}
 \beta:=2+\mu-d(q-1)-2b.
\end{equation}
Moreover,
\begin{equation}\label{eq:pcEbound}
 \sup_{0\le \tau<1}
 \left\{
 (1-\tau)^\beta\|\nabla\psi(\tau)\|_{L^2}^2
 +
 \mathcal I_{b,\mu}(\psi(\tau))
 \right\}
 <\infty,
\end{equation}
with \( \mathcal I_{b,\mu}\) defined as in \eqref{eq:weightedE}.
\end{proposition}

\begin{proof}
We first perform the calculation for smooth solutions, the general
energy solution follows by the standard approximation argument. Indicate 
\begin{equation}\label{eq:lambdat}
 \lambda:=1-\tau=\frac1{1+t}.
\end{equation}
A direct differentiation of
\eqref{eq:inhCon} gives the identity
\begin{equation}\label{eq:PLT}
 \left(i\partial_\tau+\Delta_y\right)\psi(\tau,y)
 =
 \lambda^{-\frac d2-2}
 e^{-\frac{i|y|^2}{4\lambda}}
 \left(i\partial_t+\Delta_x\right)
 \phi\left(t,\frac y\lambda\right),
\end{equation}
with the inverse relation
\[
 \phi\left(t,\frac{\zeta}{\lambda}\right)
 =
 \lambda^{d/2}
 e^{\frac{i|\zeta|^2}{4\lambda}}
 \psi(\tau,\zeta).
\]
Hence, after the change of variables \(z=\zeta/\lambda\),
\begin{align}
 \left(
 I_\mu*(|\cdot|^{-b}|\phi(t)|^q)
 \right)\left(\frac y\lambda\right)
 \notag\\
=
 \int_{\mathbb R^d}
 I_\mu\left(\frac{y-\zeta}{\lambda}\right)
 \left|\frac{\zeta}{\lambda}\right|^{-b}
 \left|
 \phi\left(t,\frac{\zeta}{\lambda}\right)
 \right|^q
 \lambda^{-d}\,d\zeta
 =
 \lambda^{\frac{dq}{2}+b-\mu}
 \left(
 I_\mu*(|\cdot|^{-b}|\psi(\tau)|^q)
 \right)(y),
 \label{eq:PCS}
\end{align}
because
\[
 I_\mu\left(\frac{y-\zeta}{\lambda}\right)
 =
 \lambda^{d-\mu}I_\mu(y-\zeta),
 \qquad
 \left|\frac{\zeta}{\lambda}\right|^{-b}
 =
 \lambda^b|\zeta|^{-b}.
\]
Moreover,
\[
 \left|\frac y\lambda\right|^{-b}
 \left|\phi\left(t,\frac y\lambda\right)\right|^{q-2}
 \phi\left(t,\frac y\lambda\right)
 =
 \lambda^{b+\frac{d(q-1)}2}
 e^{\frac{i|y|^2}{4\lambda}}
 |y|^{-b}|\psi|^{q-2}\psi.
\]
By a combination of this identity with
\eqref{eq:PCS}, we obtain
\begin{equation}
 \mathcal N_{b,\mu}(\phi)
 \left(t,\frac y\lambda\right)
 =
 \lambda^{\,2b-\mu+dq-\frac d2}
 e^{\frac{i|y|^2}{4\lambda}}
 \mathcal N_{b,\mu}(\psi)(\tau,y).
\end{equation}
Substitution into \eqref{eq:PLT} yields
\[
 i\partial_\tau\psi+\Delta_y\psi
 =
 (1-\tau)^{-\beta}\mathcal N_{b,\mu}(\psi),
\]
which proves \eqref{eq:pcpsieq}. By Proposition~\ref{prop:FMP} in
Appendix~\ref{sec:BLWS}, we have
\(
 \phi\in C\bigl(\mathbb R;\Sigma\bigr).
\)
Moreover, the change of variables \(y=\lambda x\), together with
direct differentiation of \eqref{eq:inhCon}, gives
\begin{equation}\label{eq:pcSI}
 \|\psi(\tau)\|_{L^2}
 =
 \|\phi(t)\|_{L^2},
 \qquad
 \|y\psi(\tau)\|_{L^2}
 =
 \lambda\|x\phi(t)\|_{L^2},
\end{equation}
and
\begin{equation}
 \|\nabla_y\psi(\tau)\|_{L^2}
 =
 \lambda^{-1}
 \left\|
  \nabla_x\phi(t)
  -
  \frac{i\lambda x}{2}\phi(t)
 \right\|_{L^2}.
\end{equation}
Hence \(\psi(\tau)\in\Sigma\) for every \(0\le\tau<1\).
Since the scaling operators and the multipliers
\(e^{-ia|x|^2}\), \(a\in\mathbb R\), appearing in
\eqref{eq:inhCon} form strongly continuous families on \(\Sigma\), the
fact that \(\phi\in C(\mathbb R;\Sigma)\) implies
\(
 \psi\in C\bigl([0,1);\Sigma\bigr).
\)
It remains to prove the uniform bound. It is easy to see that
\begin{equation}\label{eq:PID}
 \frac{d}{d\tau}\mathcal I_{b,\mu}(\psi(\tau))
 =
 2q\,\operatorname{Re}
 \int_{\mathbb R^d}
 \mathcal N_{b,\mu}(\psi)
 \overline{\partial_\tau\psi}\,dx.
\end{equation}
Multiplying \eqref{eq:pcpsieq} by
\(\lambda^\beta\overline{\partial_\tau\psi}\), 
taking the real part and using \eqref{eq:PID} together with
integration by parts in the spatial variable, we obtain,
\[
 \lambda^\beta
 \frac{d}{d\tau}\|\nabla\psi(\tau)\|_{L^2}^2
 +
 \frac1q
 \frac{d}{d\tau}\mathcal I_{b,\mu}(\psi(\tau))
 =
 0.
\]
Therefore, for
\[
 Q(\tau)
 :=
 \lambda^\beta\|\nabla\psi(\tau)\|_{L^2}^2
 +
 \frac1q\mathcal I_{b,\mu}(\psi(\tau)),
\]
one has
\begin{equation}
 Q'(\tau)
 =
 -\beta\lambda^{\beta-1}
 \|\nabla\psi(\tau)\|_{L^2}^2
 \le0.
\end{equation}
Thus \(Q(\tau)\le Q(0)\) for every \(0\le\tau<1\). We recall that the solution of \eqref{eq:ME} satisfies, for all \( t\in\mathbb R\),
the energy conservation law in \eqref{eq:CE2},
because the potential energy is nonnegative, one has
\(
 \mathcal I_{b,\mu}(\phi_0)
 \le
 2qE(\phi_0)
 <\infty.
\)
Moreover, at \(\tau=0\), \eqref{eq:inhCon} gives
\[
 \psi(0,x)
 =
 e^{-\frac{i|x|^2}{4}}\phi_0(x).
\]
For \(\phi_0\in\Sigma\), one gets
\[
 x\psi(0,x)
 =
 e^{-\frac{i|\cdot|^2}{4}}x\phi_0
 \in L^2(\R^d),
\]
while
\[
 \nabla\psi(0,x)
 =
 e^{-\frac{i|\cdot|^2}{4}}
 \left(
  \nabla\phi_0-\frac{i x}{2}\phi_0
 \right)
 \in L^2(\R^d).
\]
Thus \(\psi(0,x)\in\Sigma\). One also obtain
\(
 \mathcal I_{b,\mu}(\psi(0))
 =
 \mathcal I_{b,\mu}(\phi_0)
 <\infty.
\)
Consequently,
\[
 Q(0)
 =
 \|\nabla\psi(0)\|_{L^2}^2
 +
 \frac1q\mathcal I_{b,\mu}(\psi(0))
 <\infty.
\]
Since both terms
in \(Q\) are nonnegative,
\[
 (1-\tau)^\beta\|\nabla\psi(\tau)\|_{L^2}^2
 +
 \mathcal I_{b,\mu}(\psi(\tau))
 \le
 qQ(0),
\]
uniformly for \(0\le\tau<1\). This proves
\eqref{eq:pcEbound}.
\end{proof}
The preceding proposition immediately infers the following uniform
bounds for the pseudoconformal solution.
\begin{corollary}
\label{cor:psia}
Under the hypotheses of Proposition~\ref{prop:BPCP1}, the solution  \(\psi\in C([0,1);\Sigma)\) of \eqref{eq:pcpsieq} with initial data \(\psi(0,x)\in \Sigma\) satisfies
\begin{equation}\label{eq:psia}
 \sup_{0\le\tau<1}\|\psi(\tau)\|_{L^2}<\infty,
 \qquad
 \sup_{0\le\tau<1}
 \mathcal I_{b,\mu}(\psi(\tau))<\infty,
 \qquad
 (1-\tau)^\beta
 \|\nabla\psi(\tau)\|_{L^2}^2
 \lesssim1.
\end{equation}
\end{corollary}

\begin{proof}
The first bound follows from the conservation of mass in \eqref{eq:CM2} and
\eqref{eq:pcSI}, while the remaining two follow from
\eqref{eq:pcEbound}.
\end{proof}

In the homogeneous case, the same bounds also give a decay estimate
for the nonlocal potential functional in the original variables.
\begin{equation}\label{eq:GAP}
 \|\nabla\psi(\tau)\|_{L^2}
 \lesssim
 (1-\tau)^{-\beta/2},
 \qquad
 0\le\tau<1.
\end{equation}

\begin{corollary}
Assume \(b=0\), let \(\nu\) be defined by
\eqref{eq:IBN}. Under the hypotheses of Proposition~\ref{prop:BPCP1}, the solution 
\(
\phi\in C\bigl(\mathbb R;H^1(\mathbb R^d)\bigr)
\)
of \eqref{eq:ME} with initial data \(\phi_0\in\Sigma\), fulfils 
\begin{equation}\label{eq:BJID}
 \mathcal I_{0,\mu}(\phi(t))
 \lesssim
 (1+t)^{-\nu},
 \qquad
 t\ge0,
\end{equation}
with \( \mathcal I_{0,\mu}\) defined as in \eqref{eq:weightedE}
\end{corollary}

\begin{proof}
The pseudoconformal transform in \eqref{eq:inhCon} gives
\[
 |\phi(t,x)|
 =
 \lambda^{d/2}|\psi(\tau,\lambda x)|.
\]
Consequently, the changes of variables
\(y=\lambda x\) and \(w=\lambda z\) yield
\begin{align*}
 \mathcal I_{0,\mu}(\phi(t))
 =
 \lambda^{dq}
 \lambda^{d-\mu}
 \lambda^{-2d}
 \mathcal I_{0,\mu}(\psi(\tau))
 =
 \lambda^{d(q-1)-\mu}
 \mathcal I_{0,\mu}(\psi(\tau))
 =
 \lambda^\nu
 \mathcal I_{0,\mu}(\psi(\tau)).
\end{align*}
By \eqref{eq:pcEbound}, we get
\[
 \sup_{0\le\tau<1}
 \mathcal I_{0,\mu}(\psi(\tau))
 <\infty.
\]
and therefore
\[
 \mathcal I_{0,\mu}(\phi(t))
 \lesssim
 \lambda^\nu
 =
 (1+t)^{-\nu}.
\]
\end{proof}
We derive below the nonlinear estimates needed
for the \(L^2\)-scattering argument. Throughout the remainder of this
section, \(\psi\) denotes the pseudoconformal transform defined in
\eqref{eq:inhCon}. It satisfies \eqref{eq:pcpsieq} and the bounds
\eqref{eq:psia}. Moreover, \eqref{eq:frf} gives
\(
 0<\beta<1.
\)
For
nonnegative measurable functions \(F,G\), let us set
\begin{equation}
 B_{b,\mu}(F,G)
 :=
 B\bigl(|\cdot|^{-b}F,|\cdot|^{-b}G\bigr).
\end{equation}
In particular,  we notice that
\begin{equation}\label{eq:WIAB}
 \mathcal I_{b,\mu}(f)
 =
 B_{b,\mu}(|f|^q,|f|^q).
\end{equation}

\begin{lemma}
Let \(F,G\ge0\) be measurable.

\begin{enumerate}
\item[\rm(i)]
If
\[
 B_{b,\mu}(F,F)<\infty,
 \qquad
 B_{b,\mu}(G,G)<\infty,
\]
then we have
\begin{equation}\label{eq:BC}
 B_{b,\mu}(F,G)^2
 \le
 B_{b,\mu}(F,F)\,B_{b,\mu}(G,G).
\end{equation}

\item[\rm(ii)]
Assume \(2b<d+\mu\). If
\(\eta\in C_c^\infty(\R^d)\), then one gets
\begin{equation}\label{eq:TIF}
 \mathcal I_{b,\mu}(\eta)<\infty.
\end{equation}
\end{enumerate}
\end{lemma}

\begin{proof}
We first prove \eqref{eq:BC}. For \(n\ge2\), define
\[
 F_n
 :=
 \min\{F,n\}
 \mathbf 1_{\{n^{-1}<|x|<n\}},
 \qquad
 G_n
 :=
 \min\{G,n\}
 \mathbf 1_{\{n^{-1}<|x|<n\}}.
\]
Then
\[
 \lim_{n\to\infty}F_n(x)=F(x),
 \qquad
 \lim_{n\to\infty}G_n(x)=G(x)
\]
almost everywhere in \(\R^d\). Moreover,
\(
 |x|^{-b}F_n,\ |x|^{-b}G_n
 \in L^{p_\mu}(\R^d),
\)
since \(F_n\) and \(G_n\) are bounded and supported in 
\(\{n^{-1}<|x|<n\}\). By the definition of \(B_{b,\mu}\) and the inequality \eqref{eq:BBC}, one obtains
\[
 B_{b,\mu}(F_n,G_n)^2
 \le
 B_{b,\mu}(F_n,F_n)
 B_{b,\mu}(G_n,G_n).
\]
Since the kernel is nonnegative, the monotone convergence theorem gives
\[
 \lim_{n\to\infty}
 B_{b,\mu}(F_n,G_n)
 =
 B_{b,\mu}(F,G).
\]
and
\[
 \lim_{n\to\infty}
 B_{b,\mu}(F_n,F_n)
 =
 B_{b,\mu}(F,F),
 \qquad
 \lim_{n\to\infty}
 B_{b,\mu}(G_n,G_n)
 =
 B_{b,\mu}(G,G).
\]
Passing to the limit as \(n\to\infty\), we obtain
\[
 B_{b,\mu}(F,G)^2
 \le
 B_{b,\mu}(F,F)B_{b,\mu}(G,G),
\]
which demonstrates \eqref{eq:BC}. To prove \eqref{eq:TIF}, let
\(\eta\in C_c^\infty(\R^d)\), and set
\(
 h_\eta(x)
 :=
 |x|^{-b}|\eta(x)|^q.
\)
Recall that
\(
 p_\mu=2d/(d+\mu).
\)
It is straightforward that the assumption \(2b<d+\mu\) is equivalent to
\(
 bp_\mu<d.
\)
Choose \(R>0\) such that
\(\operatorname{supp}\eta\subset B_R(0)\), then we earn
\[
 \|h_\eta\|_{L^{p_\mu}}^{p_\mu}
 \le
 \|\eta\|_{L^\infty}^{qp_\mu}
 \int_{B_R(0)}|x|^{-bp_\mu}\,dx
 <\infty.
\]
Therefore \(h_\eta\in L^{p_\mu}(\R^d)\), and by
\eqref{eq:BHLS},
\[
 \mathcal I_{b,\mu}(\eta)
 =
 B(h_\eta,h_\eta)
 \lesssim
 \|h_\eta\|_{L^{p_\mu}}^2
 <\infty.
\]
This proves \eqref{eq:TIF}.
\end{proof}
We also need the following estimate for the nonlinear term
\(\mathcal I_{b,\mu}\).

\begin{lemma}
\label{lem:WNP1}
Let \(u,v\) be measurable functions such that
\begin{equation}\label{eq:bou1}
 \mathcal I_{b,\mu}(u)<\infty,
 \qquad
 \mathcal I_{b,\mu}(v)<\infty.
\end{equation}
Then one has
\begin{equation}\label{eq:WNP2}
 \left|
 \left\langle
 \mathcal N_{b,\mu}(u),v
 \right\rangle
 \right|
 \le
 \mathcal I_{b,\mu}(u)^{1-\frac1{2q}}
 \mathcal I_{b,\mu}(v)^{\frac1{2q}}.
\end{equation}
\end{lemma}

\begin{proof}
Define the nonnegative, possibly extended-valued, measurable
function
\[
 A_u(x)
 :=
 |x|^{-b}
 \bigl(
 I_\mu*(|\cdot|^{-b}|u|^q)
 \bigr)(x).
\]
Since all the integrands involved are nonnegative, Tonelli's theorem
gives, initially as identities in \([0,\infty]\),
\begin{equation}\label{eq:ASI}
 \int_{\R^d}A_u(x)|u(x)|^q\,dx
 =
 \mathcal I_{b,\mu}(u),
\end{equation}
and
\begin{equation}\label{eq:AMI}
 \int_{\R^d}A_u(x)|v(x)|^q\,dx
 =
 B_{b,\mu}(|u|^q,|v|^q).
\end{equation}
Indeed,
\begin{align*}
 \int_{\R^d}A_u(x)|v(x)|^q\,dx
 =
 \int_{\R^d}
 |x|^{-b}|v(x)|^q
 \left(
  \int_{\R^d}
  \frac{|y|^{-b}|u(y)|^q}
  {|x-y|^{d-\mu}}
  \,dy
 \right)dx
 \\
 =
 \iint_{\R^d\times\R^d}
 \frac{|u(y)|^q|v(x)|^q}
 {|x|^b|x-y|^{d-\mu}|y|^b}
 \,dy\,dx
 =
 B_{b,\mu}(|u|^q,|v|^q).
\end{align*}
By \eqref{eq:WIAB}, \eqref{eq:BC} and \eqref{eq:bou1}, we achieve
\begin{align}
 B_{b,\mu}(|u|^q,|v|^q)
 \le
 B_{b,\mu}(|u|^q,|u|^q)^{1/2}
 B_{b,\mu}(|v|^q,|v|^q)^{1/2}
 =
 \mathcal I_{b,\mu}(u)^{1/2}
 \mathcal I_{b,\mu}(v)^{1/2}
 <\infty.
 \label{eq:MIF}
\end{align}
Consequently, both integrals in
\eqref{eq:ASI} and
\eqref{eq:AMI} are finite. In particular, \(A_u<\infty\) almost everywhere on the sets where
\(u\neq0\) or \(v\neq0\). Thus the quantity
\(
 A_u|u|^{q-2}u\,\overline v
\)
is well defined almost everywhere. Introduce the positive measure
\(
 d\nu_u(x):=A_u(x)\,dx.
\)
Since \(q>1\), the H\"older's inequality on the measure space
\((\R^d,\nu_u)\) therefore gives
\begin{align}
 \int_{\R^d}
 A_u|u|^{q-1}|v|\,dx
 &=
 \int_{\R^d}
 |u|^{q-1}|v|\,d\nu_u
 \le
 \left(
  \int_{\R^d}|u|^q\,d\nu_u
 \right)^{\frac{q-1}{q}}
 \left(
  \int_{\R^d}|v|^q\,d\nu_u
 \right)^{\frac1q}
 \notag\\
 &=
 \mathcal I_{b,\mu}(u)^{\frac{q-1}{q}}
 B_{b,\mu}(|u|^q,|v|^q)^{\frac1q}.
 \label{eq:ANP}
\end{align}
The right-hand side of the above equality is finite by
\eqref{eq:MIF}, this gives
\[
 \int_{\R^d}
 \left|
 \mathcal N_{b,\mu}(u)\overline v
 \right|\,dx
 =
 \int_{\R^d}
 A_u|u|^{q-1}|v|\,dx
 <\infty,
\]
and so 
\[
 \left\langle
 \mathcal N_{b,\mu}(u),v
 \right\rangle
 :=
 \int_{\mathbb R^d}
 \mathcal N_{b,\mu}(u)\overline v\,dx
\]
is therefore well defined. From
\eqref{eq:ANP} 
\eqref{eq:MIF}, 
\begin{align*}
 \left|
 \left\langle
 \mathcal N_{b,\mu}(u),v
 \right\rangle
 \right|
 \le
 \mathcal I_{b,\mu}(u)^{\frac{q-1}{q}}
 \left(
  \mathcal I_{b,\mu}(u)^{1/2}
  \mathcal I_{b,\mu}(v)^{1/2}
 \right)^{\frac1q}
 \\
 =
 \mathcal I_{b,\mu}(u)^{
  \frac{q-1}{q}+\frac1{2q}}
 \mathcal I_{b,\mu}(v)^{\frac1{2q}}
 =
 \mathcal I_{b,\mu}(u)^{1-\frac1{2q}}
 \mathcal I_{b,\mu}(v)^{\frac1{2q}}.
\end{align*}
This shows \eqref{eq:WNP2}.
\end{proof}

\begin{lemma}\label{lem:NFT}
Let \(\eta\in C_c^\infty(\R^d)\). Under the hypotheses of Proposition~\ref{prop:BPCP1}, the solution  \(\psi\in C([0,1);\Sigma)\) with initial data \(\psi(0,x)\in \Sigma\) satisfies, for every
\(0\le\tau<1\),
\begin{equation}\label{eq:FTB}
 \left|
 \left\langle
 \mathcal N_{b,\mu}(\psi(\tau)),\eta
 \right\rangle
 \right|
 \le
 \mathcal I_{b,\mu}(\psi(\tau))^{1-\frac1{2q}}
 \mathcal I_{b,\mu}(\eta)^{\frac1{2q}}.
\end{equation}
In particular,
\begin{equation}\label{eq:FTU}
 \sup_{0\le\tau<1}
 \left|
 \left\langle
 \mathcal N_{b,\mu}(\psi(\tau)),\eta
 \right\rangle
 \right|
 <\infty.
\end{equation}
\end{lemma}

\begin{proof}
Fix \(0\le\tau<1\). By \eqref{eq:psia}, we get
\(
 \mathcal I_{b,\mu}(\psi(\tau))<\infty.
\)
Moreover, since \(\eta\in C_c^\infty(\R^d)\),
\eqref{eq:TIF} gives
\(
 \mathcal I_{b,\mu}(\eta)<\infty.
\)
Hence Lemma~\ref{lem:WNP1}, applied with
\(
 u=\psi(\tau), \,
 v=\eta,
\)
yields
\begin{align*}
 \left|
 \left\langle
 \mathcal N_{b,\mu}(\psi(\tau)),\eta
 \right\rangle
 \right|
 &\le
 \mathcal I_{b,\mu}(\psi(\tau))^{1-\frac1{2q}}
 \mathcal I_{b,\mu}(\eta)^{\frac1{2q}}.
\end{align*}
Since \(\eta\) is fixed and
\(\mathcal I_{b,\mu}(\eta)<\infty\), this implies
\[
 \left|
 \left\langle
 \mathcal N_{b,\mu}(\psi(\tau)),\eta
 \right\rangle
 \right|
 \lesssim_\eta
 \mathcal I_{b,\mu}(\psi(\tau))^{1-\frac1{2q}},
\]
which demonstrates \eqref{eq:FTB}. Finally, let
\[
 M
 :=
 \sup_{0\le s<1}
 \mathcal I_{b,\mu}(\psi(s)).
\]
From \eqref{eq:psia}, we have \(M<\infty\). Therefore, one gets
\[
 \left|
 \left\langle
 \mathcal N_{b,\mu}(\psi(\tau)),\eta
 \right\rangle
 \right|
 \le
 M^{1-\frac1{2q}}
 \mathcal I_{b,\mu}(\eta)^{\frac1{2q}}
\]
for every \(0\le\tau<1\). The above inequality brings to
\[
 \sup_{0\le\tau<1}
 \left|
 \left\langle
 \mathcal N_{b,\mu}(\psi(\tau)),\eta
 \right\rangle
 \right|
 \le
 M^{1-\frac1{2q}}
 \mathcal I_{b,\mu}(\eta)^{\frac1{2q}}
 <\infty.
\]
This proves \eqref{eq:FTU}.
\end{proof}

\section{\(L^{2}\)-scattering}\label{sec:IEB}

This section displays \(L^2\)-scattering by showing that the
pseudoconformal profile \(\psi(\tau)\) converges strongly in \(L^2\)
as \(\tau\to1^-\). Only positive times are considered below, the corresponding conclusions for negative times follow by applying
the same arguments to
\(
 \widetilde\phi(t,x):=\overline{\phi(-t,x)}.
\)

\begin{proposition}
\label{prop:WLFR}
Under the hypotheses of Proposition~\ref{prop:BPCP1}, Let \(\psi\in C([0,1);\Sigma)\) be the solution of \eqref{eq:pcpsieq} with initial data \(\psi(0,x)\in \Sigma\). Then
there exists unique \(\psi_{+}\in L^2(\mathbb R^d)\) such that
\begin{equation}\label{eq:WLV}
 \psi(\tau)\rightharpoonup\psi_{+}
 \quad\text{in }L^2(\mathbb R^d),
\quad\text{as }  \tau\to1^-.
\end{equation}
\end{proposition}

\begin{proof}
Let \(0\le \tau<s<1\). For every \(\eta\in C_c^\infty(\R^d)\), the weak formulation of
\eqref{eq:pcpsieq}, integrated between \(\tau\) and \(s\), gives
\begin{equation}\label{eq:WFI1}
 \ip{\psi(s)-\psi(\tau)}{\eta}
 =
 -\ii\int_\tau^s
 \ip{\nabla\psi(r)}{\nabla\eta}\,dr
 -
 \ii\int_\tau^s
 (1-r)^{-\beta}
 \ip{\Ncal_{b,\mu}(\psi(r))}{\eta}\,dr.
\end{equation}
For the linear part, \eqref{eq:GAP} gives
\begin{align}
 \abs{\int_\tau^s \ip{\nabla \psi(r)}{\nabla\eta}\,dr}
 &\le
 \|\nabla\eta\|_{L^2}\int_\tau^s \|\nabla \psi(r)\|_{L^2}\,dr
\lesssim
 \|\nabla\eta\|_{L^2}\int_\tau^s (1-r)^{-\beta/2}\,dr.
 \label{eq:LEB}
\end{align}
For the nonlinear part, Lemma \ref{lem:NFT} yields
\[
 \abs{\ip{\Ncal_{b,\mu}(\psi(r))}{\eta}}\lesssim 1
 \qquad (0\le r<1).
\]
Hence
\begin{align}
 \abs{\int_\tau^s (1-r)^{-\beta}\ip{\Ncal_{b,\mu}(\psi(r))}{\eta}\,dr}
 &\lesssim_\eta
 \int_\tau^s (1-r)^{-\beta}\,dr.
 \label{eq:NEB}
\end{align}
Combining \eqref{eq:WFI1}, \eqref{eq:LEB}, and \eqref{eq:NEB}, we obtain,
for every \(0\le\tau<s<1\) and
\(\eta\in C_c^\infty(\R^d)\),
\[
 \left|
 \ip{\psi(s)-\psi(\tau)}{\eta}
 \right|
 \lesssim_\eta
 \int_\tau^s
 \left(
  (1-r)^{-\beta/2}
  +(1-r)^{-\beta}
 \right)\,dr.
\]
Since \(0<\beta<1\), it follows that
\[
 \lim_{\substack{\tau,s\to1^-\\ \tau<s}}
 \left|
 \ip{\psi(s)-\psi(\tau)}{\eta}
 \right|
 =0.
\]
Hence the limit
\[
 \Lambda(\eta)
 :=
 \lim_{\tau\to1^-}
 \ip{\psi(\tau)}{\eta}
\]
exists for every \(\eta\in C_c^\infty(\R^d)\). By \eqref{eq:psia}, one obtains
\[
 |\Lambda(\eta)|
 \le
 C_\psi\|\eta\|_{L^2},
 \qquad
 C_\psi
 :=
 \sup_{0\le\tau<1}\|\psi(\tau)\|_{L^2}
 <\infty.
\]
Hence \(\Lambda\) extends uniquely to a bounded conjugate-linear
functional on \(L^2(\R^d)\). By the Riesz representation theorem,
there exists a unique \(\psi_+\in L^2(\R^d)\) such that
\(
 \Lambda(g)=\ip{\psi_+}{g},
\)
with
\(
 g\in L^2(\R^d).
\)
Finally, the density of \(C_c^\infty(\R^d)\) in \(L^2(\R^d)\),
together with the uniform bound in \eqref{eq:psia}, extends the
convergence to every \(g\in L^2(\R^d)\). Therefore
\(
 \psi(\tau)\rightharpoonup \psi_+
 \)
as \(\tau\to1^-\), which is \eqref{eq:WLV}.
\end{proof}

We now move from weak to strong convergence.

\begin{lemma}\label{lem:NVB1}
Under the hypotheses of Proposition~\ref{prop:BPCP1}, let \(\psi\)
be the corresponding pseudoconformal transform, and set
\begin{equation}\label{eq:NVB3}
 M_\psi
 :=
 \sup_{0\le s<1}
 \mathcal I_{b,\mu}(\psi(s)).
\end{equation}
Then \(M_\psi<\infty\). Moreover, for every
\(0\le r,\tau<1\),
\(
 \mathcal N_{b,\mu}(\psi(r))
 \overline{\psi(\tau)}
 \in L^1(\R^d),
\)
and
\begin{equation}\label{eq:NVB2}
 \left|
 \left\langle
 \mathcal N_{b,\mu}(\psi(r)),\psi(\tau)
 \right\rangle
 \right|
 \le
 \mathcal I_{b,\mu}(\psi(r))^{1-\frac1{2q}}
 \mathcal I_{b,\mu}(\psi(\tau))^{\frac1{2q}}
 \le
 M_\psi.
\end{equation}
\end{lemma}

\begin{proof}
By \eqref{eq:psia}, one has \(M_\psi<\infty\).  In particular,
for every \(0\le r,\tau<1\),
\[
 \mathcal I_{b,\mu}(\psi(r))<\infty,
 \qquad
 \mathcal I_{b,\mu}(\psi(\tau))<\infty.
\]
Lemma~\ref{lem:WNP1}, applied with
\(
 u=\psi(r),
 \)
  \(
  v=\psi(\tau),
\)
shows that
\(
 \mathcal N_{b,\mu}(\psi(r))
 \overline{\psi(\tau)}
 \in L^1(\R^d)
\)
and gives the first inequality in \eqref{eq:NVB2}.  The second
inequality follows from \eqref{eq:NVB3}, since
\[
 \left(1-\frac1{2q}\right)+\frac1{2q}=1.
\]
\end{proof}
The above estimate provides the bound on the nonlinear term needed to pass
from the weak convergence in \eqref{eq:WLV} to strong
\(L^2\)-convergence.

\begin{theorem}\label{thm:L2S}
Assume \eqref{eq:ass} and \eqref{eq:frf}. Let
\(\phi_0\in\Sigma\), and let
\(
 \phi\in C\bigl(\mathbb R;H^1(\mathbb R^d)\bigr)
\)
be the corresponding global solution of \eqref{eq:ME} with
\(\phi_0\in \Sigma\). Let \(\psi\) be the pseudoconformal transform of
\(\phi\) defined by \eqref{eq:inhCon}. Then there exists
\(\psi_+\in L^2(\mathbb R^d)\) such that
\begin{equation}\label{eq:L2S1}
 \lim_{\tau\to1^-}
 \|\psi(\tau)-\psi_+\|_{L^2}
 =0.
\end{equation}
Consequently, there exists \(\phi_+\in L^2(\mathbb R^d)\) such that
\begin{equation}\label{eq:L2S1A}
 \lim_{t\to+\infty}
 \left\|
  \phi(t)-e^{it\Delta}\phi_+
 \right\|_{L^2}
 =0.
\end{equation}
\end{theorem}

\begin{proof}
By Proposition \ref{prop:WLFR}, there exists \(\psi_+\in L^2(\R^d)\) such that
\(
 \psi(\tau)\rightharpoonup \psi_+
 \)
as \(\tau\to1^-\). It remains to prove strong convergence. Let \(0\le \tau<s<1\). Using \eqref{eq:pcpsieq}, we establish
\begin{equation}\label{eq:PWV}
 \ip{\psi(s)-\psi(\tau)}{\psi(\tau)}
 =
 -i\int_\tau^s \ip{\nabla \psi(r)}{\nabla \psi(\tau)}\,dr
 -i\int_\tau^s (1-r)^{-\beta}\ip{\Ncal_{b,\mu}(\psi(r))}{\psi(\tau)}\,dr.
\end{equation}
Here the integral involving the nonlinear term is well defined by
Lemma~\ref{lem:NVB1}. Using \eqref{eq:GAP}, \eqref{eq:NVB2}  and \eqref{eq:PWV}, we find, for \(0\le\tau<s<1\),
\begin{align}
  \left|
 \ip{\psi(s)-\psi(\tau)}{\psi(\tau)}
 \right|
 &\lesssim
 (1-\tau)^{-\beta/2}
 \int_\tau^s(1-r)^{-\beta/2}\,dr
 +
 \int_\tau^s(1-r)^{-\beta}\,dr
 \lesssim
 (1-\tau)^{1-\beta}.
 \label{eq:NVS}
\end{align}
Fixing \(0\le\tau<1\) and letting \(s\to1^-\), the weak convergence
\(\psi(s)\rightharpoonup\psi_+\) in \(L^2(\R^d)\) gives
\[
 \left|
 \ip{\psi_+-\psi(\tau)}{\psi(\tau)}
 \right|
 \lesssim
 (1-\tau)^{1-\beta}.
\]
Since \(0<\beta<1\), it follows that
\begin{equation}\label{eq:SPC1}
 \lim_{\tau\to1^-}
 \ip{\psi(\tau)-\psi_+}{\psi(\tau)}
 =0.
\end{equation}
On the other hand, the weak convergence
\(\psi(\tau)\rightharpoonup\psi_+\) also yields
\begin{equation}\label{eq:SPC2}
 \lim_{\tau\to1^-}
 \ip{\psi(\tau)-\psi_+}{\psi_+}
 =0.
\end{equation}
Therefore,
\begin{align*}
 \lim_{\tau\to1^-}
 \|\psi(\tau)-\psi_+\|_{L^2}^2
 &=
 \lim_{\tau\to1^-}
 \left[
 \Re\ip{\psi(\tau)-\psi_+}{\psi(\tau)}
 -
 \Re\ip{\psi(\tau)-\psi_+}{\psi_+}
 \right]
 =0,
\end{align*}
which proves \eqref{eq:L2S1}. It remains to relate the compactified limit to a scattering state in
the original variables.  Let \( \lambda\) and \(\tau\) as in \eqref{eq:lambdat}. The same computation used  to prove \eqref{eq:PLT}, applied to
the free solution \(e^{\ii t\Delta}f\), shows that, for every
\(f\in L^2(\R^d)\),
\begin{equation}\label{eq:PCF1}
 \lambda^{-d/2}
 e^{-\frac{\ii|y|^2}{4\lambda}}
 \bigl(e^{\ii t\Delta}f\bigr)
 \left(\frac{y}{\lambda}\right)
 =
 e^{\ii\tau\Delta}
 \left(
  e^{-\frac{\ii|\cdot|^2}{4}}f
 \right)(y).
\end{equation}
Indeed, the transformed function solves the free
Schr\"odinger equation in the variable \(\tau\) and, at \(\tau=0\),
has initial value \(e^{-\ii|\cdot|^2/4}f\).  The identity follows first
for smooth data and then for all \(L^2\)-data by density and unitarity. For each fixed \(t\ge0\), take
\(
 f=e^{-\ii t\Delta}\phi(t)
\)
in \eqref{eq:PCF1}.  Since \(e^{\ii t\Delta}f=\phi(t)\), the definition
\eqref{eq:inhCon} gives
\begin{equation}\label{eq:PCI2}
 e^{-\ii t\Delta}\phi(t)
 =
 e^{\frac{\ii|\cdot|^2}{4}}
 e^{-\ii\tau\Delta}\psi(\tau)
\end{equation}
in \(L^2(\R^d)\). Define
\begin{equation}\label{eq:PCS3}
 \phi_+
 :=
 e^{\frac{\ii|\cdot|^2}{4}}
 e^{-\ii\Delta}\psi_+.
\end{equation}
Multiplication by \(e^{\ii|x|^2/4}\) and the free Schr\"odinger group
are unitary on \(L^2\), so \(\phi_+\in L^2(\R^d)\).  Moreover, by
\eqref{eq:PCI2} and the strong continuity of the
free Schr\"odinger group on \(L^2\) one has
\begin{align*}
 \limsup_{t\to+\infty}
 \norm{e^{-\ii t\Delta}\phi(t)-\phi_+}_{L^2}
 =
 \limsup_{\tau\to1^-}
 \norm{
  e^{-\ii\tau\Delta}\psi(\tau)
  -
  e^{-\ii\Delta}\psi_+
 }_{L^2}
\\
 \le
 \lim_{\tau\to1^-}
 \norm{\psi(\tau)-\psi_+}_{L^2}
 +
 \lim_{\tau\to1^-}
 \norm{
  \bigl(e^{-\ii\tau\Delta}-e^{-\ii\Delta}\bigr)\psi_+
 }_{L^2}
 =0.
\end{align*}
Therefore, one obtains
\[
 \lim_{t\to+\infty}
 \norm{\phi(t)-e^{\ii t\Delta}\phi_+}_{L^2}
 =
 \lim_{t\to+\infty}
 \norm{e^{-\ii t\Delta}\phi(t)-\phi_+}_{L^2}
 =0.
\]
Thus \eqref{eq:L2S1A} holds and \(\phi\) scatters in \(L^2(\R^d)\).

\end{proof}

\section{Scattering in \(H^{1}\): proof of the Theorem \ref{thm:MH1}}
\label{sec:BPUC}

In this section we prove Theorem~\ref{thm:MH1} by providing a self contained adaptation of the
scattering argument of
\cite{BurqGeorgievTzvetkovVisciglia}. We treat only the case \(t\>0\) since
the case \(t<0\) follows by applying the same argument to
\(\overline{\phi(-t,x)}\).

\begin{lemma}
\label{lem:BPGV1}
Under the hypotheses of
Proposition~\ref{prop:BPCP1}, one has, for every \(t\ge0\), the following
\begin{equation}\label{eq:BPGVR} 
 \left\|
 \nabla\phi(t)
 -
 \frac{i x}{2(1+t)}\phi(t)
 \right\|_{L^2}
 \lesssim
 (1+t)^{-1+\frac{\beta}{2}}.
\end{equation}
In particular, we get 
\begin{equation}\label{eq:BPGV2}
 \lim_{t\to+\infty}
 \left\|
 \nabla\phi(t)
 -
 \frac{i x}{2(1+t)}\phi(t)
 \right\|_{L^2}
 =0,
\end{equation}
and
\begin{equation}\label{eq:BPRMP}
 \sup_{t\ge0}
 \left\|
 \frac{x}{1+t}\phi(t)
 \right\|_{L^2}
 <\infty.
\end{equation}
\end{lemma}

\begin{proof}
For \(t\ge0\), choose \( \tau\), \( \lambda\) as in \eqref{eq:lambdat} and 
\(
 y:=\lambda x.
\)
This enhances to
\[
 t=\frac{\tau}{1-\tau},
 \qquad
 x=\frac{y}{\lambda}.
\]
By \eqref{eq:inhCon}, we achieve
\begin{equation}
 \psi(\tau,y)
 =
 \lambda^{-d/2}
 e^{-\frac{i|y|^2}{4\lambda}}
 \phi\left(t,\frac{y}{\lambda}\right).
\end{equation}
For each
\(j=1,\dots,d\), we can write
\begin{align*}
 \partial_{y_j}\psi(\tau,y)
 &=
 \lambda^{-d/2}
 e^{-\frac{i|y|^2}{4\lambda}}
 \left[
 -\frac{i y_j}{2\lambda}\phi(t,x)
 +
 \frac1\lambda\partial_{x_j}\phi(t,x)
 \right]
 =
 \lambda^{-d/2-1}
 e^{-\frac{i|y|^2}{4\lambda}}
 \left[
 \partial_{x_j}\phi(t,x)
 -
 \frac{i y_j}{2}\phi(t,x)
 \right].
\end{align*}
Thus, we obtain the identity
\begin{equation}\label{eq:BPGTE}
 \nabla_y\psi(\tau,y)
 =
 \lambda^{-d/2-1}
 e^{-\frac{i|y|^2}{4\lambda}}
 \left(
 \nabla_x\phi(t,x)
 -
 \frac{i\lambda x}{2}\phi(t,x)
 \right).
\end{equation}
The above identity is valid in the weak \(L^2\)-sense under the regularity
assumptions of
Proposition~\ref{prop:BPCP1}. Indeed, for each fixed
\(t\ge0\), we have
\(
 \phi(t)\in\Sigma,
\)
and hence
\(
 \nabla\phi(t)\in L^2,
 x\phi(t)\in L^2.
\)
The right-hand side of
\eqref{eq:BPGTE} therefore belongs to
\(L^2\). The identity may equivalently be obtained first for smooth
functions and then extended to \(\Sigma\) by density.

The identity \eqref{eq:BPGTE} gives
\begin{align*}
 \|\nabla\psi(\tau)\|_{L^2_y}^2
 =
 \int_{\mathbb R^d}
 \lambda^{-d-2}
 \left|
 \nabla\phi(t,x)
 -
 \frac{i\lambda x}{2}\phi(t,x)
 \right|^2dy
 \\
 =
 \lambda^{-d-2}\lambda^d
 \int_{\mathbb R^d}
 \left|
 \nabla\phi(t,x)
 -
 \frac{i\lambda x}{2}\phi(t,x)
 \right|^2dx
 =
 \lambda^{-2}
 \left\|
 \nabla\phi(t)
 -
 \frac{i\lambda x}{2}\phi(t)
 \right\|_{L^2_x}^2,
\end{align*}
which reads as
\begin{equation}\label{eq:BPGI}
 \left\|
 \nabla\phi(t)
 -
 \frac{i x}{2(1+t)}\phi(t)
 \right\|_{L^2}
 =
 \lambda\|\nabla\psi(\tau)\|_{L^2}.
\end{equation}
The conformal-energy estimate in
\eqref{eq:pcEbound}, implies
\begin{equation}\label{eq:BPTGR}
 \|\nabla\psi(\tau)\|_{L^2}
 \lesssim
 \lambda^{-\beta/2}.
\end{equation}
Substituting
\eqref{eq:BPTGR} into
\eqref{eq:BPGI}, we find
\begin{align*}
 \left\|
 \nabla\phi(t)
 -
 \frac{i x}{2(1+t)}\phi(t)
 \right\|_{L^2}
 &\lesssim
 \lambda^{1-\beta/2}
 =
 (1+t)^{-1+\beta/2}.
\end{align*}
This proves
\eqref{eq:BPGVR}. The hypotheses of
Proposition~\ref{prop:BPCP1} give
\(
 0<\beta<1.
\)
In particular,
\[
 1-\frac{\beta}{2}>0.
\]
It follows that
\begin{equation}\label{eq:BPSDL}
 \lim_{t\to+\infty}
 (1+t)^{-1+\beta/2}
 =0.
\end{equation}
A combination of
\eqref{eq:BPGVR} and
\eqref{eq:BPSDL} bring to
\[
 0
 \le
 \limsup_{t\to+\infty}
 \left\|
 \nabla\phi(t)
 -
 \frac{i x}{2(1+t)}\phi(t)
 \right\|_{L^2}
 \le0.
\]
Therefore,
\[
 \lim_{t\to+\infty}
 \left\|
 \nabla\phi(t)
 -
 \frac{i x}{2(1+t)}\phi(t)
 \right\|_{L^2}
 =0,
\]
which proves
\eqref{eq:BPGV2}. We finally prove
\eqref{eq:BPRMP}. Define
\[
 A(t)
 :=
 \nabla\phi(t)
 -
 \frac{i x}{2(1+t)}\phi(t).
\]
Consequently, one has the inequality
\begin{align}
 \left\|
 \frac{x}{1+t}\phi(t)
 \right\|_{L^2}
 =
 2
 \left\|
 \frac{i x}{2(1+t)}\phi(t)
 \right\|_{L^2}
 \notag\\
 =
 2\|\nabla\phi(t)-A(t)\|_{L^2}
 \le
 2\|\nabla\phi(t)\|_{L^2}
 +
 2\|A(t)\|_{L^2}.
 \label{eq:BPMFG}
\end{align}
By  \eqref{eq:CE3}, there exists a
constant \(C_E<\infty\), such that
\begin{equation}\label{eq:BPPGU}
 \sup_{t\ge0}\|\nabla\phi(t)\|_{L^2}
 \le C_E.
\end{equation}
On the other hand, by
\eqref{eq:BPGVR}, one gets
\[
 \|A(t)\|_{L^2}
 \lesssim
 (1+t)^{-1+\beta/2}.
\]
Since \(1-\beta/2>0\), it follows that
\begin{equation}\label{eq:BPAU}
 \sup_{t\ge0}\|A(t)\|_{L^2}<\infty.
\end{equation}
The, an use of 
\eqref{eq:BPMFG},
\eqref{eq:BPPGU} and
\eqref{eq:BPAU}, gives
\[
 \sup_{t\ge0}
 \left\|
 \frac{x}{1+t}\phi(t)
 \right\|_{L^2}
 <\infty.
\]
This proves
\eqref{eq:BPRMP}.
\end{proof}

The following result is the Choquard analogue of
\cite[Lemma~4.2]{BurqGeorgievTzvetkovVisciglia}.
We include the proof because our pseudoconformal variables
lead first to the scale \(1+t\).

\begin{lemma}
\label{lem:BPT}
Under the hypotheses of
Proposition~\ref{prop:BPCP1}, one has
\begin{equation}\label{eq:BPTO}
 \lim_{R\to\infty}\,
 \limsup_{t\to+\infty}
 \int_{|x|>R(1+t)}
 \frac{|x|^2}{(1+t)^2}
 |\phi(t,x)|^2\,dx
 =0.
\end{equation}
\end{lemma}

\begin{proof}
Let
\[
 \tau=\frac{t}{1+t},
 \qquad
 \lambda=1-\tau=\frac1{1+t},
 \qquad
 y=\lambda x.
\]
By the definition \eqref{eq:inhCon} of the pseudoconformal transform
and the change of variables \(y=x/(1+t)\),
\begin{equation}\label{eq:BPMC}
 \int_{|x|>R(1+t)}
 \frac{|x|^2}{(1+t)^2}
 |\phi(t,x)|^2\,dx
 =
 \int_{|y|>R}
 |y|^2|\psi(\tau,y)|^2\,dy.
\end{equation}
Since \(t\to+\infty\) if and only if \(\tau\to1^-\), it is therefore
enough to prove
\begin{equation}\label{eq:BPTTG}
 \lim_{R\to\infty}
 \limsup_{\tau\to1^-}
 \int_{|y|>R}
 |y|^2|\psi(\tau,y)|^2\,dy
 =0.
\end{equation}
By \eqref{eq:pcSI} and \eqref{eq:BPRMP}, one has
\begin{equation}\label{eq:BPYUB}
 \sup_{0\le\tau<1}
 \|y\psi(\tau)\|_{L^2}
 <\infty.
\end{equation}
Choose a radial function
\(\chi\in C^\infty(\R^d;[0,\infty))\) such that
\[
 \chi(y)=0
 \quad\text{for }|y|\le\frac12,
 \qquad
 \chi(y)=|y|
 \quad\text{for }|y|\ge1,
\]
and
\[
 \chi(y)\lesssim |y|,
 \qquad
 |\nabla\chi(y)|\lesssim1.
\]
For \(R>0\), set
\[
 \chi_R(y):=R\chi(y/R),
 \qquad
 M_R(\tau):=
 \int_{\R^d}\chi_R(y)^2|\psi(\tau,y)|^2\,dy.
\]
Then
\begin{equation}\label{eq:BBcsi}
\chi_R(y)=|y|
 \quad\text{for }|y|\ge R,
 \qquad
 \chi_R(y)\lesssim |y|,
 \qquad
 \|\nabla\chi_R\|_{L^\infty}\lesssim1,
\end{equation}
with constants independent of \(R\). In particular,
\eqref{eq:BPYUB} gives
\begin{equation}\label{eq:BPMRU}
 \sup_{R>0}\sup_{0\le\tau<1}M_R(\tau)^{1/2}
 <\infty.
\end{equation}
By an use of
\(
 \overline{\psi}\,\mathcal N_{b,\mu}(\psi)\in\mathbb R
\)
and
\(
 \psi\in C([0,1);\Sigma),
\)
we obtain
\begin{equation}\label{eq:BPMRD}
 M_R'(\tau)
 =
 4\int_{\R^d}
 \chi_R\nabla\chi_R\cdot
 \operatorname{Im}(\overline{\psi}\nabla\psi)\,dy
\end{equation}
for almost every \(\tau\in(0,1)\). Futhermore, by \eqref{eq:BPYUB} and \eqref{eq:psia}, we attain
\begin{equation}\label{eq:BPMRDU}
 |M_R'(\tau)|
 \lesssim
 (1-\tau)^{-\beta/2},
\end{equation}
with a constant independent of \(R\). Since \(0<\beta<1\), a direct calculations infers
\[
 \lim_{\tau_0\to1^-}
 \int_{\tau_0}^1
 (1-s)^{-\beta/2}\,ds
 =0.
\]
Fix \(\varepsilon>0\), and choose
\(\tau_0\in(0,1)\) so close to \(1\) that
\begin{equation}\label{eq:BPTSM1}
 C\int_{\tau_0}^1
 (1-s)^{-\beta/2}\,ds
 <\frac{\varepsilon}{2},
\end{equation}
where \(C\) is the hidden constant in \eqref{eq:BPMRDU}. Since \(\psi(\tau_0)\in\Sigma\), we have
\(|y|\psi(\tau_0)\in L^2\). Moreover, for every \(y\in\R^d\), 
\[
 \lim_{R\to\infty}\chi_R(y)=0,
\]
so \( \chi_R(y)\lesssim |y|\) in \eqref{eq:BBcsi} and the dominated convergence theorem give
\[
 \lim_{R\to\infty}M_R(\tau_0)=0.
\]
Hence there exists \(R_\varepsilon>0\) such that
\[
 M_{R_\varepsilon}(\tau_0)
 <\frac{\varepsilon}{2}.
\]
Integrating \eqref{eq:BPMRDU}, for every
\(\tau\in[\tau_0,1)\), brings to
\[
 M_{R_\varepsilon}(\tau)
 \le
 M_{R_\varepsilon}(\tau_0)
 +
 C\int_{\tau_0}^\tau
 (1-s)^{-\beta/2}\,ds
 <\varepsilon.
\]
Since
\(\chi_{R_\varepsilon}(y)=|y|\) for
\(|y|\ge R_\varepsilon\), it follows that
\[
 \sup_{\tau_0\le\tau<1}
 \int_{|y|>R_\varepsilon}
 |y|^2|\psi(\tau,y)|^2\,dy
 <\varepsilon.
\]
By monotonicity, the same estimate holds for every
\(R\ge R_\varepsilon\). Therefore
\[
 \lim_{R\to\infty}
 \limsup_{\tau\to1^-}
 \int_{|y|>R}
 |y|^2|\psi(\tau,y)|^2\,dy
 =0,
\]
which proves \eqref{eq:BPTTG}. Finally, \eqref{eq:BPMC} yields
\eqref{eq:BPTO}.
\end{proof}
The previous estimate is used in the next step of the
\(H^1\)-scattering argument.

\begin{proposition}
\label{prop:BPGC}
 Assume \eqref{eq:ass} and \eqref{eq:frf}. Let
\(
\phi\in C\bigl(\mathbb R;H^1(\mathbb R^d)\bigr)
\)
be a global  solution of
\eqref{eq:ME}. Let
\(
 \phi_{+}\in H^1(\mathbb R^d),
\)
and assume that
\begin{equation}\label{eq:BPMA}
 \lim_{t\to+\infty}
 \bigl\|
 |\phi(t)|-|e^{it\Delta}\phi_+|
 \bigr\|_{L^2}
 =0.
\end{equation}
Then
\begin{equation}\label{eq:BPGNL}
 \lim_{t\to+\infty}
 \|\nabla\phi(t)\|_{L^2}
 =
 \|\nabla\phi_+\|_{L^2}.
\end{equation}
\end{proposition}

\begin{proof}
The argument follows the proof of
\cite[Theorem~1.2, in particular equations~(4.6)-(4.10)]
{BurqGeorgievTzvetkovVisciglia}. See also
\cite{TzvetkovVisciglia2021} for more details. For convenience, set
\(
 \rho(t):=1+t.
\)
Lemma~\ref{lem:BPGV1} enhances to
\begin{equation}
 \lim_{t\to+\infty}
 \left\|
 \nabla\phi(t)
 -
 \frac{i x}{2\rho(t)}\phi(t)
 \right\|_{L^2}
 =0,
\end{equation}
and we achieve
\begin{align*}
 &
 \left|
 \|\nabla\phi(t)\|_{L^2}
 -
 \left\|
 \frac{x}{2\rho(t)}\phi(t)
 \right\|_{L^2}
 \right|
 =
 \left|
 \|\nabla\phi(t)\|_{L^2}
 -
 \left\|
 \frac{i x}{2\rho(t)}\phi(t)
 \right\|_{L^2}
 \right|
 \le
 \left\|
 \nabla\phi(t)
 -
 \frac{i x}{2\rho(t)}\phi(t)
 \right\|_{L^2}.
\end{align*}
Therefore, one gets
\begin{equation}\label{eq:BPGMDL}
 \lim_{t\to+\infty}
 \left|
 \|\nabla\phi(t)\|_{L^2}
 -
 \left\|
 \frac{x}{2\rho(t)}\phi(t)
 \right\|_{L^2}
 \right|
 =0.
\end{equation}
It is then enough to establish
\begin{equation}\label{eq:BPRML}
 \lim_{t\to+\infty}
 \left\|
 \frac{x}{2\rho(t)}\phi(t)
 \right\|_{L^2}
 =
 \|\nabla\phi_+\|_{L^2}.
\end{equation}
Fix \(R>0\) now, and define
\[
 \Omega_R(t)
 :=
 \{x\in\mathbb R^d:|x|<R\rho(t)\}
\]
and
\[
 w_t(x)
 :=
 \frac{|x|}{2\rho(t)}.
\]
For every \(x\in\Omega_R(t)\),
\begin{equation}\label{eq:BPIWB}
 0\le w_t(x)<\frac R2.
\end{equation}
Let
\(
 F_t:=e^{it\Delta}\phi_+.
\)
We use the Fourier transform
\[
 \widehat h(\xi)
 :=
 (2\pi)^{-d/2}
 \int_{\R^d}e^{-ix\cdot\xi}h(x)\,dx,
\]
and for \(t>0\), define
\begin{equation}\label{eq:BPFMD}
 G_t(x)
 :=
 (2it)^{-d/2}
 e^{\frac{i|x|^2}{4t}}
 \widehat{\phi_+}\left(\frac{x}{2t}\right).
\end{equation}
By the standard \(L^2\)-asymptotic formula for the free
Schr\"odinger group (see, for example,
\cite[Eq.~(4.8)]{BurqGeorgievTzvetkovVisciglia} or also \cite{Dollard1971,ReedSimonII}), one arrives at
\begin{equation}\label{eq:BPFA}
 \lim_{t\to+\infty}
 \|F_t-G_t\|_{L^2}
 =0.
\end{equation}
Consequently, we obtain
\begin{equation}\label{eq:BPFMA}
 \lim_{t\to+\infty}
 \bigl\|
 |F_t|-|G_t|
 \bigr\|_{L^2}
 =0.
\end{equation}
We now compare \(\phi(t)\) and \(G_t\) in \(\Omega_R(t)\). By
\eqref{eq:BPIWB}, one achieves
\begin{align*}
 &
 \left\|
 w_t\bigl(|\phi(t)|-|F_t|\bigr)
 \right\|_{L^2(\Omega_R(t))}
 \le
 \frac R2
 \bigl\|
 |\phi(t)|-|F_t|
 \bigr\|_{L^2}.
\end{align*}
Thus \eqref{eq:BPMA} gives
\begin{equation}
 \lim_{t\to+\infty}
 \left\|
 w_t\bigl(|\phi(t)|-|F_t|\bigr)
 \right\|_{L^2(\Omega_R(t))}
 =0.
\end{equation}
Similarly, by
\eqref{eq:BPFMA}, we deduce
\begin{equation}
 \lim_{t\to+\infty}
 \left\|
 w_t\bigl(|F_t|-|G_t|\bigr)
 \right\|_{L^2(\Omega_R(t))}
 =0.
\end{equation}
An application of the triangle inequality therefore yields
\begin{equation}\label{eq:BPWMC}
 \lim_{t\to+\infty}
 \left\|
 w_t\bigl(|\phi(t)|-|G_t|\bigr)
 \right\|_{L^2(\Omega_R(t))}
 =0.
\end{equation}
We introduce
\[
 a_R(t)
 :=
 \|w_t\phi(t)\|_{L^2(\Omega_R(t))}
\]
and
\[
 b_R(t)
 :=
 \|w_tG_t\|_{L^2(\Omega_R(t))}.
\]
An use of the reverse
triangle inequality and
\eqref{eq:BPWMC} give
\begin{equation}\label{eq:BPARBRL}
 \lim_{t\to+\infty}
 |a_R(t)-b_R(t)|
 =0.
\end{equation}
First observe that
\[
 b_R(t)
 \le
 \frac R2\|G_t\|_{L^2},
\]
then by using \eqref{eq:BPFMD} and the change of
variables
\begin{equation}\label{eq:ChVar}
 \xi=\frac{x}{2t},
 \qquad
 dx=(2t)^d\,d\xi,
\end{equation}
we compute
\begin{align*}
 \|G_t\|_{L^2}^2
 =
 \int_{\mathbb R^d}
 (2t)^{-d}
 \left|
 \widehat{\phi_+}\left(\frac{x}{2t}\right)
 \right|^2dx
 =
 \int_{\mathbb R^d}
 |\widehat{\phi_+}(\xi)|^2\,d\xi
 =
 \|\phi_+\|_{L^2}^2.
\end{align*}
Hence we have
\begin{equation}\label{eq:BPBRB}
 b_R(t)\le\frac R2\|\phi_+\|_{L^2}
\end{equation}
for every \(t>0\). Let
\[
 \delta_R(t):=|a_R(t)-b_R(t)|.
\]
From the elementary inequality
\[
 a_R(t)\le b_R(t)+\delta_R(t),
\]
we obtain
\begin{align*}
 |a_R(t)^2-b_R(t)^2|
 &=
 \delta_R(t)\bigl(a_R(t)+b_R(t)\bigr)
 \le
 \delta_R(t)\bigl(2b_R(t)+\delta_R(t)\bigr).
\end{align*}
Using \eqref{eq:BPARBRL} and
\eqref{eq:BPBRB}, we conclude that
\begin{equation}\label{eq:BPSC}
 \lim_{t\to+\infty}
 |a_R(t)^2-b_R(t)^2|
 =0.
\end{equation}
It remains to compute \(b_R(t)^2\). By definition, we can write
\begin{align*}
 b_R(t)^2
 &=
 \int_{|x|<R\rho(t)}
 \frac{|x|^2}{4\rho(t)^2}
 (2t)^{-d}
 \left|
 \widehat{\phi_+}\left(\frac{x}{2t}\right)
 \right|^2dx.
\end{align*}
Using again \eqref{eq:ChVar} one has
\[
 \frac{|x|^2}{4\rho(t)^2}
 =
 \left(\frac{t}{\rho(t)}\right)^2|\xi|^2,
\]
Therefore,
\begin{equation}\label{eq:BPMME}
 b_R(t)^2
 =
 \left(\frac{t}{\rho(t)}\right)^2
 \int_{|\xi|<\frac{R\rho(t)}{2t}}
 |\xi|^2|\widehat{\phi_+}(\xi)|^2\,d\xi.
\end{equation}
Since \(\rho(t)=1+t\), it is easy to see that
\begin{equation}\label{eq:BPRL}
 \lim_{t\to+\infty}
 \frac{t}{\rho(t)}
 =1
\end{equation}
and
\begin{equation}
 \lim_{t\to+\infty}
 \frac{R\rho(t)}{2t}
 =
 \frac R2.
\end{equation}
Moreover, the fact that \(\phi_+\in H^1(\mathbb R^d)\), guarantees
\(
 |\xi|^2|\widehat{\phi_+}(\xi)|^2
 \in L^1(\mathbb R^d).
\)
For almost every \(\xi\), one accomplishes
\[
 \lim_{t\to+\infty}
 \mathbf 1_{\left\{
 |\xi|<\frac{R\rho(t)}{2t}
 \right\}}
 =
 \mathbf 1_{\{|\xi|<R/2\}}.
\]
 Hence the dominated
convergence theorem gives
\begin{equation}\label{eq:BPMBL}
 \lim_{t\to+\infty}
 \int_{|\xi|<\frac{R\rho(t)}{2t}}
 |\xi|^2|\widehat{\phi_+}(\xi)|^2\,d\xi
 =
 \int_{|\xi|<R/2}
 |\xi|^2|\widehat{\phi_+}(\xi)|^2\,d\xi.
\end{equation}
A combined use of
\eqref{eq:BPMME},
\eqref{eq:BPRL}, and
\eqref{eq:BPMBL}, gives
\begin{equation}\label{eq:BPMIL}
 \lim_{t\to+\infty}
 b_R(t)^2
 =
 \int_{|\xi|<R/2}
 |\xi|^2|\widehat{\phi_+}(\xi)|^2\,d\xi.
\end{equation}
Since
\[
 a_R(t)^2
 =
 \int_{|x|<R(1+t)}
 \frac{|x|^2}{4(1+t)^2}
 |\phi(t,x)|^2\,dx,
\]
equations
\eqref{eq:BPSC} and
\eqref{eq:BPMIL} yield
\begin{equation}\label{eq:BPIML}
 \lim_{t\to+\infty}
 \int_{|x|<R(1+t)}
 \frac{|x|^2}{4(1+t)^2}
 |\phi(t,x)|^2\,dx
 =
 \int_{|\xi|<R/2}
 |\xi|^2|\widehat{\phi_+}(\xi)|^2\,d\xi
\end{equation}
for every fixed \(R>0\). Define now
\[
 Q(t)
 :=
 \left\|
 \frac{x}{2(1+t)}\phi(t)
 \right\|_{L^2}^2.
\]
For \(R>0\), decompose
\begin{equation}\label{eq:dec1}
 Q(t)=I_R(t)+E_R(t),
\end{equation}
where
\[
 I_R(t)
 :=
 \int_{|x|\le R(1+t)}
 \frac{|x|^2}{4(1+t)^2}
 |\phi(t,x)|^2\,dx
\]
and
\[
 E_R(t)
 :=
 \int_{|x|>R(1+t)}
 \frac{|x|^2}{4(1+t)^2}
 |\phi(t,x)|^2\,dx.
\]
Let us define
\[
 A_R
 :=
 \int_{|\xi|<R/2}
 |\xi|^2|\widehat{\phi_+}(\xi)|^2\,d\xi.
\]
The identity \eqref{eq:BPIML} leads to
\begin{equation}\label{eq:BPIRL}
 \lim_{t\to+\infty}I_R(t)=A_R.
\end{equation}
Lemma~\ref{lem:BPT} gives
\begin{equation}\label{eq:BPERT}
 \lim_{R\to\infty}
 \limsup_{t\to+\infty}
 E_R(t)
 =0,
\end{equation}
We observe that, for every \(R>0\),
\(
 Q(t)\ge I_R(t).
\)
Therefore, we have
\[
 \liminf_{t\to+\infty}Q(t)
 \ge A_R.
\]
The monotone convergence theorem brings to
\begin{equation}\label{eq:BPARL}
 \lim_{R\to\infty}A_R
 =
 \int_{\mathbb R^d}
 |\xi|^2|\widehat{\phi_+}(\xi)|^2\,d\xi,
\end{equation}
and hence we see that
\begin{equation}\label{eq:BPQL1}
 \liminf_{t\to+\infty}Q(t)
 \ge
 \int_{\mathbb R^d}
 |\xi|^2|\widehat{\phi_+}(\xi)|^2\,d\xi.
\end{equation}
Using now \eqref{eq:dec1} and
\eqref{eq:BPIRL}, we obtain
\[
 \limsup_{t\to+\infty}Q(t)
 \le
 A_R+
 \limsup_{t\to+\infty}E_R(t).
\]
Letting \(R\to\infty\), and utilizing
\eqref{eq:BPERT} and
\eqref{eq:BPARL}, gives
\begin{equation}\label{eq:BPQL2}
 \limsup_{t\to+\infty}Q(t)
 \le
 \int_{\mathbb R^d}
 |\xi|^2|\widehat{\phi_+}(\xi)|^2\,d\xi.
\end{equation}
By a combination of
\eqref{eq:BPQL1} and
\eqref{eq:BPQL2}, we obtain
\begin{equation}
 \lim_{t\to+\infty}Q(t)
 =
 \int_{\mathbb R^d}
 |\xi|^2|\widehat{\phi_+}(\xi)|^2\,d\xi.
\end{equation}
The Plancherel's theorem implies
\[
 \int_{\mathbb R^d}
 |\xi|^2|\widehat{\phi_+}(\xi)|^2\,d\xi
 =
 \|\nabla\phi_+\|_{L^2}^2,
\]
and thus one achieves
\[
 \lim_{t\to+\infty}
 \left\|
 \frac{x}{2(1+t)}\phi(t)
 \right\|_{L^2}^2
 =
 \|\nabla\phi_+\|_{L^2}^2.
\]
This proves
\eqref{eq:BPRML}. Finally, one can write
\begin{align*}
 &
 \left|
 \|\nabla\phi(t)\|_{L^2}
 -
 \|\nabla\phi_+\|_{L^2}
 \right|
\le
 \left|
 \|\nabla\phi(t)\|_{L^2}
 -
 \left\|
 \frac{x}{2(1+t)}\phi(t)
 \right\|_{L^2}
 \right|
+
 \left|
 \left\|
 \frac{x}{2(1+t)}\phi(t)
 \right\|_{L^2}
 -
 \|\nabla\phi_+\|_{L^2}
 \right|.
\end{align*}
The limit of the first term is zero by
\eqref{eq:BPGMDL}, and the limit of
the second term is zero by
\eqref{eq:BPRML}. Consequently we arrive at
\[
 \lim_{t\to+\infty}
 \left|
 \|\nabla\phi(t)\|_{L^2}
 -
 \|\nabla\phi_+\|_{L^2}
 \right|
 =0.
\]
This proves
\eqref{eq:BPGNL}.
\end{proof}

Having established \(L^2\)-scattering in the preceding section, we now
show that the corresponding scattering state belongs to \(H^1\) and
that the convergence holds in \(H^1\). More precisely

\begin{proposition}\label{prop:H1S}
 Assume \eqref{eq:ass} and \eqref{eq:frf}. Let \(\phi_0\in\Sigma\), and let
\(
\phi\in C\bigl(\mathbb R;H^1(\mathbb R^d)\bigr)
\)
be the corresponding global defocusing solution of
\eqref{eq:ME}. Assume that there exist
\(\phi_{\pm}\in L^2(\R^d)\) such that
\begin{equation}\label{eq:BPLS}
 \lim_{t\to\pm\infty}
 \|\phi(t)-e^{it\Delta}\phi_+\|_{L^2}
 =0.
\end{equation}
Then \(\phi_{\pm}\in H^1(\R^d)\), and
\begin{equation}\label{eq:BPHS}
 \lim_{t\to+\infty}
 \|\phi(t)-e^{it\Delta}\phi_+\|_{H^1}
 =0.
\end{equation}
\end{proposition}

\begin{proof}
The argument is the same as in
\cite[Section~3]{BurqGeorgievTzvetkovVisciglia}. We sketch it for aim of completeness. Set
\[
 v(t):=e^{-it\Delta}\phi(t).
\]
By the unitarity of the free Schr\"odinger propagator on \(L^2\),
\eqref{eq:BPLS} is equivalent to
\begin{equation}\label{eq:BPVLS}
 \lim_{t\to+\infty}
 \|v(t)-\phi_+\|_{L^2}
 =0.
\end{equation}
By the conservation laws \eqref{eq:CM2}, \eqref{eq:CE2} and since \(e^{-it\Delta}\) is unitary on \(H^1\), it follows that
\[
 \sup_{t\ge0}\|v(t)\|_{H^1}<\infty.
\]
Hence, by weak compactness in \(H^1\), every sequence
\((t_n)_{n\ge1}\) satisfying
\(
 t_n\to+\infty,
\)
as \(n\to\infty\),
has a subsequence converging weakly in \(H^1\). The strong
\(L^2\)-convergence \eqref{eq:BPVLS} shows that every
such weak limit is necessarily equal to \(\phi_+\). Therefore,
\begin{equation}\label{eq:BPWH}
 \phi_+\in H^1(\R^d),
 \qquad
 v(t)\rightharpoonup\phi_+
 \quad\text{in }H^1(\R^d)
\end{equation}
as \(t\to+\infty\). Moreover, the triangle inequality and \eqref{eq:BPLS} give
\begin{equation}
 \lim_{t\to+\infty}
 \bigl\|
 |\phi(t)|-|e^{it\Delta}\phi_+|
 \bigr\|_{L^2}
 =0.
\end{equation}
Proposition~\ref{prop:BPGC} therefore yields
\begin{equation}
 \lim_{t\to+\infty}
 \|\nabla\phi(t)\|_{L^2}
 =
 \|\nabla\phi_+\|_{L^2}.
\end{equation}

Since the free flow commutes with \(\nabla\) and is unitary on \(L^2\),
\[
 \|\nabla v(t)\|_{L^2}
 =
 \|\nabla\phi(t)\|_{L^2}.
\]
Hence
\begin{equation}\label{eq:BPVGNL}
 \lim_{t\to+\infty}
 \|\nabla v(t)\|_{L^2}
 =
 \|\nabla\phi_+\|_{L^2}.
\end{equation}
On the other hand, \eqref{eq:BPVLS} implies
\begin{equation}\label{eq:BPVLNL}
 \lim_{t\to+\infty}
 \|v(t)\|_{L^2}
 =
 \|\phi_+\|_{L^2}.
\end{equation}
From
\eqref{eq:BPVGNL} and
\eqref{eq:BPVLNL}, we acquire
\begin{align*}
 \lim_{t\to+\infty}\|v(t)\|_{H^1}^2
 =
 \lim_{t\to+\infty}
 \left(
 \|v(t)\|_{L^2}^2
 +
 \|\nabla v(t)\|_{L^2}^2
 \right)
 =
 \|\phi_+\|_{L^2}^2
 +
 \|\nabla\phi_+\|_{L^2}^2
 =
 \|\phi_+\|_{H^1}^2.
\end{align*}
It follows that
\begin{equation}\label{eq:BPVHNL}
 \lim_{t\to+\infty}
 \|v(t)\|_{H^1}
 =
 \|\phi_+\|_{H^1}.
\end{equation}
Finally, \(H^1(\R^d)\) is a Hilbert space. Therefore, the weak
convergence \eqref{eq:BPWH}, together with the norm
convergence \eqref{eq:BPVHNL}, implies
\begin{equation}\label{eq:BPVHS}
 \lim_{t\to+\infty}
 \|v(t)-\phi_+\|_{H^1}
 =0.
\end{equation}
Indeed, we earns
\begin{align*}
 \|v(t)-\phi_+\|_{H^1}^2
 &=
 \|v(t)\|_{H^1}^2
 +
 \|\phi_+\|_{H^1}^2
 -
 2\operatorname{Re}
 \langle v(t),\phi_+\rangle_{H^1},
\end{align*}
and
\[
 \lim_{t\to+\infty}
 \langle v(t),\phi_+\rangle_{H^1}
 =
 \|\phi_+\|_{H^1}^2
\]
by \eqref{eq:BPWH}. As a result we get
\[
 \lim_{t\to+\infty}
 \|v(t)-\phi_+\|_{H^1}^2
 =0,
\]
which is equivalent to
\eqref{eq:BPVHS}. By the unitarity of \(e^{it\Delta}\) on \(H^1\),
and
\eqref{eq:BPVHS}, we conclude that
\[
 \lim_{t\to+\infty}
 \|\phi(t)-e^{it\Delta}\phi_+\|_{H^1}
 =0.
\]
This proves \eqref{eq:BPHS}.
\end{proof}

\begin{proof}[Proof of Theorem~\ref{thm:MH1}]
By Theorem~\ref{thm:L2S}, there exists
\(\phi_+\in L^2(\R^d)\) such that
\[
 \lim_{t\to+\infty}
 \|\phi(t)-e^{it\Delta}\phi_+\|_{L^2}
 =0.
\]
Proposition~\ref{prop:H1S} then gives
\(\phi_+\in H^1(\R^d)\) and
\[
 \lim_{t\to+\infty}
 \|\phi(t)-e^{it\Delta}\phi_+\|_{H^1}
 =0.
\]
The case \(t<0\) follows from the argument described
at the beginning of this section. This concludes the proof.
\end{proof}

\section{Scattering in \(\Sigma\) for the
homogeneous Choquard equation}
\label{sec:BJUS}

The preceding sections establish \(H^1\)-scattering for the
inhomogeneous Choquard equation throughout the short-range
mass-subcritical interval. We now focus on the homogeneous case
\(b=0\) and prove the stronger conclusion of scattering in the
conformal space defined in \eqref{eq:introS}. Introduce the Galilean
vector field
\begin{equation}\label{eq.GVect}
 J(t):=x+2it\nabla.
\end{equation}
The decay of the nonlocal potential given in \eqref{eq:BJID}, together with some properties of the equation in \eqref{eq:ME}, yields convergence of
\(xe^{-it\Delta}\phi(t)\) in \(L^2\).
The argument applies throughout \(q_0<q<q_{\mathrm m}\), where
\(q_0=1+X_0\) and \(X_0>0\) is the positive root of
\[
 2dX^2+(d-2\mu-2)X-\mu-2=0.
\]
Throughout this section we treat only the case \(t\ge0\). Since the equation is invariant under the transformation
\(
 \phi(t,x)
 :=\overline{\phi(-t,x)},
\)
all the corresponding statements as \(t\to-\infty\) follow by applying
the same arguments to \(\widetilde\phi\).

\subsection{Estimates for the nonlocal potential term}

We use the bilinear form \(B\) associated with the Riesz kernel
\(I_\mu\) defined in \eqref{eq:Nonl}, together with the exponent
\(p_\mu\) introduced in Lemma~\ref{lem:BRFI}. The estimates
\eqref{eq:BHLS} and \eqref{eq:BBH} apply both to the nonlinearity and
to the terms arising after differentiation, and supply the following
estimate.
\begin{lemma}
\label{lem:BIG}
Let \(f\in W^{1,r}\), assume
\begin{equation}\label{eq:assA1}
 0<\mu<d,
 \qquad
 q\ge2,
\end{equation}
and
\begin{equation}\label{eq:assA2}
 r:=\frac{2dq}{d+\mu}.
\end{equation}
Define
\[
 \mathcal I_{0,\mu}(f)
 :=
 \iint_{\mathbb R^d\times\mathbb R^d}
 I_\mu(x-y)|f(x)|^q|f(y)|^q\,dx\,dy.
\]
Then
\begin{equation}\label{eq:Bil1}
 \|\Ncal_{0,\mu}(f)\|_{W^{1,r'}}
 \lesssim
 \mathcal I_{0,\mu}(f)^{\frac{q-1}{q}}\|f\|_{W^{1,r}}.
\end{equation}
The implicit constant depends only on \(d\), \(\mu\), and \(q\).
\end{lemma}

\begin{proof} We give the proof for \(f\in C_0^\infty\), the general case follows by the
standard density argument in the Sobolev spaces. Let \(p_\mu\) and \(B\) be as introduced in
Lemma~\ref{lem:BRFI}. Since
\(
 p_\mu =r/q,
\)
Hence all the applications of \(B\)
below are covered by
Lemma~\ref{lem:BRFI}. We first estimate the \(L^{r'}\)-norm. Let
\(\varphi\in C_c^\infty(\R^d)\), and set
\(
 A:=|f|^q,
\)
and
\(
 G:=|\varphi|^q.
\)
Using the definition of \(\mathcal N_{0,\mu}\) in \eqref{eq:Nonl}, we obtain
\begin{align}
 \left|
 \int_{\R^d}
 \mathcal N_{0,\mu}(f)\overline{\varphi}\,dx
 \right|
 \le
 B\left(
  A,
  A^{\frac{q-1}{q}}G^{\frac1q}
 \right)
 =
 B\left(
  A^{\frac{q-1}{q}}A^{\frac1q},
  A^{\frac{q-1}{q}}G^{\frac1q}
 \right)
 \le
 B(A,A)^{\frac{q-1}{q}}
 B(A,G)^{\frac1q}.
\end{align}
The last inequality is indeed
\eqref{eq:BBH}, applied with weights
\(
 (q-1)/q
 \)
and
\(
1/q.
\)
By the Hardy-Littlewood-Sobolev estimate
\eqref{eq:BHLS} we earn
\begin{align}
 B(A,G)^{1/q}
 &\lesssim
 \|A\|_{L^{p_\mu}}^{1/q}
 \|G\|_{L^{p_\mu}}^{1/q}
 =
 \|f\|_{L^r}\|\varphi\|_{L^r}.
\end{align}
Since
\(
 B(A,A)=\mathcal I_{0,\mu}(f),
\)
duality gives
\begin{equation}\label{eq:BILB}
 \|\mathcal N_{0,\mu}(f)\|_{L^{r'}}
 \lesssim
 \mathcal I_{0,\mu}(f)^{\frac{q-1}{q}}
 \|f\|_{L^r}.
\end{equation}
Furthermore, because \(q\ge2\), a direct calculations enhances to
\[
 |\nabla(|f|^q)|
 \lesssim
 |f|^{q-1}|\nabla f|,
 \qquad
 |\nabla(|f|^{q-2}f)|
 \lesssim
 |f|^{q-2}|\nabla f|.\qquad \nabla(I_\mu*A)=I_\mu*(\nabla A).
\]
As a result, one achieves
\begin{align}
 |\nabla\mathcal N_{0,\mu}(f)|
 \lesssim{}&
 \bigl(
  I_\mu*(|f|^{q-1}|\nabla f|)
 \bigr)|f|^{q-1}
 +
 \bigl(I_\mu*|f|^q\bigr)
 |f|^{q-2}|\nabla f|.
 \label{eq:BGPD}
\end{align}
We now estimate these two terms separately by \(L^{r'}\)-duality. Set
\(
 D:=|\nabla f|^q.
\)
Then \(D\in L^{p_\mu}(\R^d)\) and
\(
 \|D\|_{L^{p_\mu}}
 =
 \|\nabla f\|_{L^r}^q.
\)
For the first term in
\eqref{eq:BGPD}, let
\(\varphi\in C_c^\infty(\R^d)\) and using the
positivity of the Riesz kernel, we obtain
\begin{align}\label{eq:Bi1}
 \left|
 \int_{\R^d}
 \bigl(I_\mu*(|f|^{q-1}|\nabla f|)\bigr)
 |f|^{q-1}\overline{\varphi}\,dx
 \right|
 \le
 \int_{\R^d}
 \bigl(I_\mu*(|f|^{q-1}|\nabla f|)\bigr)
 |f|^{q-1}|\varphi|\,dx
 =\notag\\
 B\bigl(
 |f|^{q-1}|\nabla f|,
 |f|^{q-1}|\varphi|
 \bigr)
 = B\left(
  A^{\frac{q-1}{q}}D^{\frac1q},
  A^{\frac{q-1}{q}}G^{\frac1q}\right)
\end{align}
By \eqref{eq:BBH}, one readily gets
\begin{align}\label{eq:Bi2}
 &B\left(
  A^{\frac{q-1}{q}}D^{\frac1q},
  A^{\frac{q-1}{q}}G^{\frac1q}
 \right)
 \le
 B(A,A)^{\frac{q-1}{q}}
 B(D,G)^{\frac1q}.
\end{align}
The bound \eqref{eq:BHLS} gives
\begin{align}\label{eq:Bi3}
 B(D,G)^{1/q}
 &\lesssim
 \|D\|_{L^{p_\mu}}^{1/q}
 \|G\|_{L^{p_\mu}}^{1/q}
 =
 \|\nabla f\|_{L^r}\|\varphi\|_{L^r}.
\end{align}
From \eqref{eq:Bi1}, \eqref{eq:Bi2} and \eqref{eq:Bi3}, we obtain
\begin{align}\label{eq:Bi1A}
 &\left|
 \int_{\R^d}
 \bigl(I_\mu*(|f|^{q-1}|\nabla f|)\bigr)
 |f|^{q-1}\overline{\varphi}\,dx
 \right|
 \lesssim
 \mathcal I_{0,\mu}(f)^{\frac{q-1}{q}}
 \|\nabla f\|_{L^r}
 \|\varphi\|_{L^r}.
\end{align}
For the second term in \eqref{eq:BGPD}, we similarly obtain
\begin{align}\label{eq:Bi4}
 \left|
 \int_{\R^d}
 \bigl(I_\mu*|f|^q\bigr)
 |f|^{q-2}|\nabla f|\overline{\varphi}\,dx
 \right|
 \le
 B\bigl(
 |f|^{q},
 |f|^{q-2}|\nabla f| |\varphi|
 \bigr)
 = B\left(
  A,
  A^{\frac{q-2}{q}}
  D^{\frac1q}
  G^{\frac1q}
 \right).
\end{align}
For \(q>2\), write
\(
 A
 =
 A^{\frac{q-2}{q}}
 A^{\frac1q}
 A^{\frac1q}.
\)
When \(q=2\), the zero-weight factor is simply omitted, in accordance
with the convention in
Lemma~\ref{lem:BRFI}. Thus
\eqref{eq:BBH} gives
\begin{align}
 &B\left(
  A,
  A^{\frac{q-2}{q}}
  D^{\frac1q}
  G^{\frac1q}
 \right)
 \le
 B(A,A)^{\frac{q-2}{q}}
 B(A,D)^{\frac1q}
 B(A,G)^{\frac1q}.
\end{align}
By the inequality
\eqref{eq:BBC}, one infers
\begin{align}
 B(A,D)^{1/q}
 \le
 B(A,A)^{\frac1{2q}}
 B(D,D)^{\frac1{2q}},
\end{align}
and
\begin{align}
 B(A,G)^{1/q}
 \le
 B(A,A)^{\frac1{2q}}
 B(G,G)^{\frac1{2q}}.
\end{align}
Moreover, \eqref{eq:BHLS} enhances to
\begin{align}\label{eq:Bi8}
 B(D,D)^{\frac1{2q}}
 &\lesssim
 \|D\|_{L^{p_\mu}}^{1/q}
 =
 \|\nabla f\|_{L^r},
 \\
 B(G,G)^{\frac1{2q}}
 &\lesssim
 \|G\|_{L^{p_\mu}}^{1/q}
 =
 \|\varphi\|_{L^r}.
\end{align}
From \eqref{eq:Bi4}--\eqref{eq:Bi8} and the identity
\[
 \frac{q-2}{q}
 +
 \frac1{2q}
 +
 \frac1{2q}
 =
 \frac{q-1}{q},
\]
we arrive at
\begin{align}\label{eq:Bi2A}
 & \left|
 \int_{\R^d}
 \bigl(I_\mu*|f|^q\bigr)
 |f|^{q-2}|\nabla f|\overline{\varphi}\,dx
 \right|
 \lesssim
 \mathcal I_{0,\mu}(f)^{\frac{q-1}{q}}
 \|\nabla f\|_{L^r}
 \|\varphi\|_{L^r}.
\end{align} 
The inequalities \eqref{eq:Bi1A} \eqref{eq:Bi2A} and a duality argument enhance to
\begin{equation}\label{eq:BIGB}
 \|\nabla\mathcal N_{0,\mu}(f)\|_{L^{r'}}
 \lesssim
 \mathcal I_{0,\mu}(f)^{\frac{q-1}{q}}
 \|\nabla f\|_{L^r}.
\end{equation}
From the inequalities
\eqref{eq:BILB} and
\eqref{eq:BIGB}, we obtain
\[
 \|\mathcal N_{0,\mu}(f)\|_{W^{1,r'}}
 \lesssim
 \mathcal I_{0,\mu}(f)^{\frac{q-1}{q}}
 \|f\|_{W^{1,r}}.
\]
that is \eqref{eq:Bil1}
\end{proof}
The following homogeneous estimate will be used below.
\begin{lemma}
Assume \(\mu, q, r\) as in \eqref{eq:assA1} and \eqref{eq:assA2}. Let \(t\in\mathbb R\setminus\{0\}\), and suppose
that
\(
 u\in L^r(\mathbb R^d),
 \)
 \(
 J(t)u\in L^r(\mathbb R^d).
\)
Then
\begin{equation}\label{eq:BJNE}
 \|J(t)\mathcal N_{0,\mu}(u)\|_{L^{r'}}
 \lesssim
 \mathcal I_{0,\mu}(u)^{\frac{q-1}{q}}
 \|J(t)u\|_{L^r}.
\end{equation}
The implicit constant depends only on \(d\), \(\mu\), and \(q\), and is
independent of \(t\) and \(u\).
\end{lemma}

\begin{proof}
Fix \(t>0\), non restrictively, and let \(M_t\) denote multiplication by the phase
\[
 M_t(x):=e^{i|x|^2/(4t)}.
\]
Since
\[
 M_t^{-1}(x)=e^{-i|x|^2/(4t)},
\]
we have
\[
 \nabla M_t^{-1}(x)
 =
 -\frac{i x}{2t}M_t^{-1}(x).
\]
Therefore, for every sufficiently regular function \(w\), we achieve
\begin{align*}
 M_t\nabla(M_t^{-1}w)
 &=
 M_t\left(
  (\nabla M_t^{-1})w+M_t^{-1}\nabla w
 \right)
 =
 M_t\left(
  -\frac{i x}{2t}M_t^{-1}w+M_t^{-1}\nabla w
 \right)
 =
 -\frac{i x}{2t}w+\nabla w.
\end{align*}
Multiplying by \(2it\), we obtain
\begin{align}
 2it\,M_t\nabla(M_t^{-1}w)
 &=
 2it\left(-\frac{i x}{2t}w+\nabla w\right)
 =
 xw+2it\nabla w
 =
 J(t)w,
\end{align}
and thus
\begin{equation}\label{eq:BJMF}
 J(t)=2it\,M_t\nabla M_t^{-1}.
\end{equation}
Set
\(
 v:=M_t^{-1}u.
\)
Since \(|M_t|=1\), multiplication by \(M_t\) and \(M_t^{-1}\) is an
isometry on every \(L^p\)-space. In particular,
\begin{equation}\label{eq:BVUL}
 \|v\|_{L^r}=\|u\|_{L^r}.
\end{equation}
Applying \eqref{eq:BJMF} to \(u\), we get
\begin{equation}\label{eq:BVGF}
 \nabla v
 =
 \frac{1}{2it}M_t^{-1}J(t)u.
\end{equation}
Because \(J(t)u\in L^r\), formula
\eqref{eq:BVGF} shows that
\(
 \nabla v\in L^r(\mathbb R^d).
\)
Together with \eqref{eq:BVUL}, this proves that
\(
 v\in W^{1,r}(\mathbb R^d).
\)
Moreover, by
\eqref{eq:BVGF}, we obtain the identity
\begin{equation}\label{eq:BMGJ}
 2t\|\nabla v\|_{L^r}
 =
 \|J(t)u\|_{L^r}.
\end{equation}
Next, since \(|v|=|u|\) pointwise,
\begin{align}
 \mathcal I_{0,\mu}(v)
 =
 \iint_{\mathbb R^d\times\mathbb R^d}
 I_\mu(x-y)|v(x)|^q|v(y)|^q\,dx\,dy
 \notag\\
 =
 \iint_{\mathbb R^d\times\mathbb R^d}
 I_\mu(x-y)|u(x)|^q|u(y)|^q\,dx\,dy
 =
 \mathcal I_{0,\mu}(u).
 \label{eq:BMEI}
\end{align}
Moreover, because \(u=M_tv\) and \(|u|=|v|\), we have
\begin{align}
 M_t^{-1}\mathcal N_{0,\mu}(u)
 =
 M_t^{-1}\Bigl[
  (I_\mu*|u|^q)|u|^{q-2}u
 \Bigr]
\notag\\ =
 M_t^{-1}\Bigl[
  (I_\mu*|v|^q)|v|^{q-2}M_tv
 \Bigr]
 =
 (I_\mu*|v|^q)|v|^{q-2}v
 =
 \mathcal N_{0,\mu}(v).
 \label{eq:BMGC}
\end{align}
We now apply 
\eqref{eq:BJMF} to \(\mathcal N_{0,\mu}(u)\). Using
\eqref{eq:BMGC}, we obtain
\begin{align}
 J(t)\mathcal N_{0,\mu}(u)
 &=
 2it\,M_t\nabla
 \left(M_t^{-1}\mathcal N_{0,\mu}(u)\right)
=
 2it\,M_t\nabla\mathcal N_{0,\mu}(v).
 \label{eq:BJMI}
\end{align}
Since \(v\in W^{1,r}\), applying the estimate
\eqref{eq:BIGB}, established in
Lemma~\ref{lem:BIG}
gives
\begin{equation}\label{eq:BGLATV}
 \|\nabla\mathcal N_{0,\mu}(v)\|_{L^{r'}}
 \lesssim
 \mathcal I_{0,\mu}(v)^{\frac{q-1}{q}}
 \|\nabla v\|_{L^r}.
\end{equation}
Multiplication by \(M_t\) is an isometry on \(L^{r'}\). Therefore, from
\eqref{eq:BJMI} and
\eqref{eq:BGLATV},
\begin{align}
 \|J(t)\mathcal N_{0,\mu}(u)\|_{L^{r'}}
 &=
 2t\|\nabla\mathcal N_{0,\mu}(v)\|_{L^{r'}}
\lesssim
 2t\,
 \mathcal I_{0,\mu}(v)^{\frac{q-1}{q}}
 \|\nabla v\|_{L^r}.
\end{align}
Finally, using the identities
\eqref{eq:BMGJ} and
\eqref{eq:BMEI}, we conclude that
\begin{align*}
 \|J(t)\mathcal N_{0,\mu}(u)\|_{L^{r'}}
 &\lesssim
 \mathcal I_{0,\mu}(u)^{\frac{q-1}{q}}
 \bigl(2t\|\nabla v\|_{L^r}\bigr)
=
 \mathcal I_{0,\mu}(u)^{\frac{q-1}{q}}
 \|J(t)u\|_{L^r}.
\end{align*}
This proves \eqref{eq:BJNE}.
\end{proof}

\subsection{The operator \(J(t)\) and \(\Sigma\)-scattering in the
homogeneous case}

In this section we restrict to the homogeneous case \(b=0\). We show
that, in the range \(q>q_0\), \(H^1\)-scattering can be upgraded
to \(\Sigma\)-scattering through an argument based on the operator
\(J(t)\) defined in \eqref{eq.GVect}.

Let \(\nu\) be defined by \eqref{eq:IBN}. We use the
Schr\"odinger-admissible pair \((a,r)\) determined by
\begin{equation}\label{eq:BJAP}
 r:=\frac{2dq}{d+\mu},
 \qquad
 \frac2a+\frac dr=\frac d2.
\end{equation}
The choice of \(r\) is equivalent to the Hardy--Littlewood--Sobolev
relation
\[
 \frac1{r'}
 =
 \frac{2q-1}{r}-\frac{\mu}{d}.
\]
The admissibility relation in
\eqref{eq:BJAP} gives
\[
 \frac2a
 =
 \frac d2-\frac dr
 =
 \frac{d(q-1)-\mu}{2q}
 =
 \frac{\nu}{2q},
\]
and hence
\begin{equation}\label{eq:b0Ja}
 a=\frac{4q}{\nu}.
\end{equation}

We also define
\begin{equation}\label{eq:kap}
 \frac1\kappa
 :=
 \frac1{a'}-\frac1a
 =
 1-\frac2a.
\end{equation}
Using \eqref{eq:b0Ja}, we obtain
\begin{equation}\label{eq:BJKF}
 \kappa
 =
 \frac{a}{a-2}
 =
 \frac{2q}{2q-\nu}.
\end{equation}

Under the mass-subcritical assumption
\[
 q<1+\frac{\mu+2}{d},
\]
one has
\[
 0<\nu<2.
\]
Indeed, \(q\ge2\) and \(0<\mu<d\) give
\[
 \nu=d(q-1)-\mu\ge d-\mu>0,
\]
whereas the upper bound on \(q\) gives \(\nu<2\).
Consequently,
\[
 a>2,
 \qquad
 1<\kappa<\infty.
\]
Recall that \(q_0=1+X_0\), where \(X_0>0\) is the positive root of
the polynomial \(P_0\) defined in
\eqref{eq:IBSP}. Setting \(X=q-1\) and using
\eqref{eq:IBN}, we compute
\begin{align*}
 \nu(2q-1)-2q
 &=
 \bigl(d(q-1)-\mu\bigr)(2q-1)-2q
 =
 2dX^2+(d-2\mu-2)X-(\mu+2)
 =
 P_0(X).
\end{align*}
Since \(q\ge2\), one has \(X=q-1>0\). Moreover,
\(P_0(0)=-(\mu+2)<0\), and \(X_0\) is the unique positive root of
\(P_0\). Therefore \(q>q_0\) is equivalent to
\begin{equation}\label{eq:BJQE}
 \nu(2q-1)>2q.
\end{equation}

The next lemma contains the time-integrability property responsible for
the strict condition \(q>q_0\).

\begin{lemma}
\label{lem:BJCI}
Assume \eqref{eq:ass2} and \eqref{eq:frf2}. Let
\(\phi_0\in\Sigma\), and let
\(
 \phi\in C\bigl(\mathbb R;H^1(\mathbb R^d)\bigr)
\)
be the corresponding global solution of
\eqref{eq:introCb0} with \(\phi(0)=\phi_0\). Then, for every \(T>0\),
\begin{equation}\label{eq:BJINT}
 \mathcal I_{0,\mu}(\phi(\cdot))^{\frac{q-1}{q}}
 \in
 L^\kappa((T,\infty)).
\end{equation}
Moreover,
\begin{equation}\label{eq:nDec}
 \lim_{T\to+\infty}
 \left\|
  \mathcal I_{0,\mu}(\phi(\cdot))^{\frac{q-1}{q}}
 \right\|_{L^\kappa((T,\infty))}
 =0.
\end{equation}
\end{lemma}

\begin{proof}
By \eqref{eq:BJID},
\begin{equation}
 \left[
 \mathcal I_{0,\mu}(\phi(t))^{\frac{q-1}{q}}
 \right]^\kappa
 \lesssim
 (1+t)^{-\nu\frac{q-1}{q}\kappa}.
\end{equation}
Using \eqref{eq:BJKF}, we have
\[
 \frac{\nu(q-1)}{q}\kappa
 =
 \frac{2\nu(q-1)}{2q-\nu}.
\]
Moreover,
\[
 \nu\frac{q-1}{q}\kappa
 =
 \nu\frac{q-1}{q}\frac{2q}{2q-\nu}
 =
 \frac{2\nu(q-1)}{2q-\nu}>1.
\]
By \eqref{eq:BJQE}, the last inequality is equivalent
to \(q>q_0\), which holds by \eqref{eq:frf2}. Thus
\(
 \mathcal I_{0,\mu}(\phi(\cdot))^{\frac{q-1}{q}}
 \in
 L^\kappa((T,\infty))
\)
for every \(T>0\). Finally, since
\(
 \mathcal I_{0,\mu}(\phi(\cdot))^{
  \kappa\frac{q-1}{q}}
 \in L^1((T,\infty)),
\)
the absolute continuity of the Lebesgue integral gives \eqref{eq:nDec}.
\end{proof}

We record the Strichartz estimates and the basic identities for
\(J(t)\) that will be used below. Recall that \(J(t)\) is defined in
\eqref{eq.GVect}. A direct computation gives
\[
 [i\partial_t,J(t)]=-2\nabla,
 \qquad
 [\Delta,J(t)]=2\nabla,
\]
and therefore
\begin{equation}\label{eq:BJC}
 [i\partial_t+\Delta,J(t)]=0.
\end{equation}
Equivalently, for every \(s,t\in\mathbb R\), one has
\begin{equation}\label{eq:BJI}
 J(t)e^{i(t-s)\Delta}
 =
 e^{i(t-s)\Delta}J(s).
\end{equation}
Taking \(s=0\) in \eqref{eq:BJI}, we obtain
\begin{equation}\label{eq:BJII}
 e^{-it\Delta}J(t)
 =
 xe^{-it\Delta}.
\end{equation}
Let \((a,r)\) be the Schr\"odinger-admissible pair defined in
\eqref{eq:BJAP}. For every interval
\(I\subset\mathbb R\), every \(t_0\in I\), and every solution of
\[
 (i\partial_t+\Delta)u=F,
 \qquad
 u(t_0)=u_0,
\]
the standard Strichartz estimate gives
\begin{equation}\label{eq:BJSS}
 \|u\|_{L^\infty_tL^2_x(I)}
 +
 \|u\|_{L^a_tL^r_x(I)}
 \lesssim
 \|u_0\|_{L^2}
 +
 \|F\|_{L^{a'}_tL^{r'}_x(I)}.
\end{equation}
We shall also use the corresponding dual estimate
\begin{equation}\label{eq:BJDS}
 \left\|
  \int_I e^{-is\Delta}F(s)\,ds
 \right\|_{L^2}
 \lesssim
 \|F\|_{L^{a'}_tL^{r'}_x(I)}.
\end{equation}
The following lemma justifies the application of
\eqref{eq:BJSS}.

\begin{lemma}
\label{lem:BJLS1}
Assume \eqref{eq:ass2} and
\(
 q<q_{\mathrm m},
\)
with \(q_{\mathrm m}\) defined as in \eqref{eq:BMa}. Let \(\phi_0\in\Sigma\), and let
\(
 \phi\in C\bigl([0,S];H^1(\mathbb R^d)\bigr)
\)
be the corresponding solution of \eqref{eq:introCb0} with \(\phi(0)=\phi_0\), where
\(S>0\). Then
\[
 J\phi
 \in
 C\bigl([0,S];L^2(\mathbb R^d)\bigr)
 \cap
 L^a\bigl((0,S);L^r(\mathbb R^d)\bigr),
\]
and
\[
 J\mathcal N_{0,\mu}(\phi)
 \in
 L^{a'}\bigl((0,S);L^{r'}(\mathbb R^d)\bigr).
\]
For every \(0\le s\le t\le S\), we have
\begin{equation}\label{eq:BJCD}
 J(t)\phi(t)
 =
 e^{i(t-s)\Delta}J(s)\phi(s)
 -
 i\int_s^t
 e^{i(t-\tau)\Delta}
 J(\tau)\mathcal N_{0,\mu}(\phi(\tau))\,d\tau.
\end{equation}
\end{lemma}

\begin{proof}
All estimates involving \(J(t)\) are understood componentwise. By Proposition~\ref{prop:FMP}, employed on \([0,S]\), we have
\(
 \phi\in C\bigl([0,S];\Sigma\bigr).
\)
As a result, one aquires
\begin{equation}
 J(t)\phi(t)
 =
 x\phi(t)+2it\nabla\phi(t)
 \in C\bigl([0,S];L^2(\mathbb R^d;\mathbb C^d)\bigr).
\end{equation}
In particular, \(J(t)\phi(t)\in L^2\) for every \(t\in[0,S]\). Choose
\[
 A_\phi(t)
 :=
 \mathcal I_{0,\mu}(\phi(t))^{\frac{q-1}{q}}.
\]
By \eqref{eq:BHLS}, we observe that
\[
 \mathcal I_{0,\mu}(\phi(t))
 \lesssim
 \|\phi(t)\|_{L^r}^{2q},
\]
and therefore
\[
 A_\phi(t)
 \lesssim
 \|\phi(t)\|_{L^r}^{2(q-1)}.
\]
The assumption
\(
 q<q_{\mathrm{m}}
\)
ensures that \(r\) lies in the \(H^1\)-Sobolev range, for \(d\geq 1\). Since
\(
 \phi\in C\bigl([0,S];H^1\bigr),
\)
we obtain
\(
%\begin{equation}\label{eq:BJMS}
 M_S
 :=
 \|A_\phi\|_{L^\infty(0,S)}
 <\infty.
%\end{equation}
\)
First, choose the initial datum is smooth and decays sufficiently
fast. Then the corresponding solution \(\phi\) is regular enough on
\([0,S]\) for the commutator identity \eqref{eq:BJC} and the commuted
Duhamel formula to be justified in the classical sense. Let \(C_{\mathrm{Str}}>0\) and \(C_{\mathrm{nl}}>0\) be constants for
which \eqref{eq:BJSS} and \eqref{eq:BJNE} hold, and set
\(
 C_*:=C_{\mathrm{Str}}C_{\mathrm{nl}}.
\)
Since \(1<\kappa<\infty\), choose \(\delta>0\) such that
\[
 C_*M_S\delta^{1/\kappa}\le\frac12.
\]
We seletct a finite partition
\(
 0=t_0<t_1<\cdots<t_N=S
\)
such that
\(
 t_{m+1}-t_m\le\delta,
\, m=0,\ldots,N-1, \, N\in \mathbb{N}.
\)
Writing
\(
 I_m:=[t_m,t_{m+1}],
\)
we then have
\begin{equation}\label{eq:BJLS2}
 C_*
 \|A_\phi\|_{L^\kappa(I_m)}
 \le\frac12
\end{equation}
for every \(m=0,\ldots,N-1\).
For a regular solution, the commutator identity
\eqref{eq:BJC} gives
\begin{equation}\label{eq:BJD1}
 (i\partial_t+\Delta)J(t)\phi
 =
 J(t)\mathcal N_{0,\mu}(\phi),
\end{equation}
and hence, on \(I_m\), one attains
\[
 J(t)\phi(t)
 =
 e^{i(t-t_m)\Delta}J(t_m)\phi(t_m)
 -
 i\int_{t_m}^t
 e^{i(t-\tau)\Delta}
 J(\tau)\mathcal N_{0,\mu}(\phi(\tau))\,d\tau.
\]
Applying \eqref{eq:BJSS}, we obtain
\begin{align}
 X_m
 &:=
 \|J\phi\|_{L^\infty_tL^2_x(I_m)}
 +
 \|J\phi\|_{L^a_tL^r_x(I_m)}
 \notag\\
 &\le
 C_{\mathrm{Str}}
 \|J(t_m)\phi(t_m)\|_{L^2}
 +
 C_{\mathrm{Str}}
 \|J\mathcal N_{0,\mu}(\phi)\|_
 {L^{a'}_tL^{r'}_x(I_m)}.
 \label{eq:BJLI1}
\end{align}
By \eqref{eq:BJNE}, one achieves
\[
 \|J(t)\mathcal N_{0,\mu}(\phi(t))\|_{L^{r'}}
 \le
 C_{\mathrm{nl}}
 A_\phi(t)\|J(t)\phi(t)\|_{L^r}
\]
for almost every \(t\in I_m\). Considering that
\begin{equation}\label{eq:Hpair}
\frac1{a'}
 =
 \frac1\kappa+\frac1a
\end{equation}
obtained by \eqref{eq:kap}, H\"older's inequality gives
\begin{equation}\label{eq:BJLI2}
 \|J\mathcal N_{0,\mu}(\phi)\|_
 {L^{a'}_tL^{r'}_x(I_m)}
 \le
 C_{\mathrm{nl}}
 \|A_\phi\|_{L^\kappa(I_m)}
 \|J\phi\|_{L^a_tL^r_x(I_m)}.
\end{equation}
From \eqref{eq:BJLI1} and \eqref{eq:BJLI2}, one gets
\[
 X_m
 \le
 C_{\mathrm{Str}}
 \|J(t_m)\phi(t_m)\|_{L^2}
 +
 C_*
 \|A_\phi\|_{L^\kappa(I_m)}X_m.
\]
By \eqref{eq:BJLS2} we arrive at
\begin{equation}\label{eq:BJLIB}
 X_m
 \le
 2C_{\mathrm{Str}}
 \|J(t_m)\phi(t_m)\|_{L^2}.
\end{equation}
As \(J(t_m)\phi(t_m)\in L^2\) for every \(m=0,\ldots,N-1\),
\eqref{eq:BJLIB} yields
\(
 J\phi\in L^a\bigl(I_m;L^r\bigr)
\)
for each \(m\). Summation over the partition therefore gives
\(
 J\phi\in L^a\bigl((0,S);L^r\bigr).
\)
Applying \eqref{eq:BJLI2} on each \(I_m\) and summing once more, we obtain
\(
 J\mathcal N_{0,\mu}(\phi)
 \in
 L^{a'}\bigl((0,S);L^{r'}\bigr).
\)
The Duhamel identity \eqref{eq:BJD1} on an arbitrary interval
\([s,t]\subset[0,S]\) guarantees \eqref{eq:BJCD} for the smooth solution. The additional smoothness assumption can now be removed by
approximation. Choose smooth initial data
\(\phi_{0,n}\to\phi_0\) in \(\Sigma\), and let \(\phi_n\) denote the
corresponding solutions. By the convergence in \(\Sigma\), together with the Hardy-Littlewood-Sobolev and Sobolev
inequalities, we acheve
\[
 \sup_n E(\phi_{0,n})<\infty.
\]
Energy conservation \eqref{eq:CE2}
therefore leads to
\[
 \sup_n
 \|A_{\phi_n}\|_{L^\infty(0,S)}
 <\infty,
\]
and therefore the same partition satisfying \eqref{eq:BJLS2} may be used for
every \(n\). Since \(J(0)=x\) and
\(\phi_{0,n}\to\phi_0\), as \(n\to \infty\) in \(\Sigma\),
\[
 \sup_n\|J(0)\phi_{0,n}\|_{L^2}
 =
 \sup_n\|x\phi_{0,n}\|_{L^2}
 <\infty.
\]
Iterating \eqref{eq:BJLIB} on the intervals \(I_m\), together with
\eqref{eq:BJLI2}, thus gives
\[
 \sup_n\left(
 \|J\phi_n\|_{L^a((0,S);L^r)}
 +
 \|J\mathcal N_{0,\mu}(\phi_n)\|_{L^{a'}((0,S);L^{r'})}
 \right)
 <\infty.
\]
Passing to the limit by weak compactness and continuous dependence on
the initial data produce
\(
 J\phi\in L^a((0,S);L^r),
 \,
 J\mathcal N_{0,\mu}(\phi)
 \in L^{a'}((0,S);L^{r'}),
\)
together with the Duhamel formula \eqref{eq:BJCD}. This is the
standard smooth-data approximation argument, see
\cite{Cazenave2003} and \cite[Lemma~A.1]{Dinh2021} for analogous
Strichartz estimates involving the Galilean vector field.
This proves the lemma.
\end{proof}
The preceding estimates, together with Lemma~\ref{lem:BJCI}, lead to
the following proposition.

\begin{proposition}\label{prop:BJDT}
Assume \eqref{eq:ass2} and \eqref{eq:frf2}. Let
\(\phi_0\in\Sigma\), and let
\(
 \phi\in C\bigl(\mathbb R;H^1(\mathbb R^d)\bigr)
\)
be the solution of
\eqref{eq:introCb0} with \(\phi(0)=\phi_0\).
Then there exists \(T_0>0\) such that, for every \(T\ge T_0\),
\[
 J\phi
 \in
 L^a\bigl((T,\infty);L^r(\mathbb R^d)\bigr),
\]
and
\[
 J\mathcal N_{0,\mu}(\phi)
 \in
 L^{a'}\bigl((T,\infty);L^{r'}(\mathbb R^d)\bigr).
\]
Moreover, there exists
\(
 \eta_+\in L^2(\mathbb R^d;\mathbb C^d)
\)
such that
\begin{equation}\label{eq:BJPML}
 \lim_{t\to+\infty}
 \left\|
  x e^{-it\Delta}\phi(t)-\eta_+
 \right\|_{L^2(\mathbb R^d;\mathbb C^d)}
 =0.
\end{equation}
\end{proposition}

\begin{proof}
All estimates involving \(J(t)\) are understood componentwise. Set
\(
 A_\phi(t)
 :=
 \mathcal I_{0,\mu}(\phi(t))^{\frac{q-1}{q}}.
\)
Let \(C_{\mathrm{Str}}>0\) and \(C_{\mathrm{nl}}>0\) be
constants for which \eqref{eq:BJSS} and \eqref{eq:BJNE} hold, and select
\(
 C_*:=C_{\mathrm{Str}}C_{\mathrm{nl}}.
\)
By \eqref{eq:nDec}, there exists \(T_0>0\) such that
\begin{equation}\label{eq:BJTS1}
 C_*
 \|A_\phi\|_{L^\kappa((T_0,\infty))}
 \le\frac12.
\end{equation}
Fix \(T\ge T_0\) and \(S>T\). By
Lemma~\ref{lem:BJLS1}, applied on \([0,S]\), one has
\[
 J\phi
 \in
 C\bigl([0,S];L^2\bigr)
 \cap
 L^a\bigl((0,S);L^r\bigr),
\qquad
 J\mathcal N_{0,\mu}(\phi)
 \in
 L^{a'}\bigl((0,S);L^{r'}\bigr),
\]
and the Duhamel formula \eqref{eq:BJCD} is valid. Since \(T\in[0,S]\)
we have
\(
 J(T)\phi(T)\in L^2(\R^d;\C^d).
\)
Applying \eqref{eq:BJSS} to \eqref{eq:BJCD} on \([T,S]\), and
retaining the \(L^a_tL^r_x\)-component, we obtain
\begin{align}
 \|J\phi\|_{L^a_tL^r_x([T,S])}
 &\le
 C_{\mathrm{Str}}
 \|J(T)\phi(T)\|_{L^2}
 +
 C_{\mathrm{Str}}
 \|J\mathcal N_{0,\mu}(\phi)\|_
 {L^{a'}_tL^{r'}_x([T,S])}.
 \label{eq:BJTS2}
\end{align}
By \eqref{eq:BJNE}, we get
\[
 \|J(t)\mathcal N_{0,\mu}(\phi(t))\|_{L^{r'}}
 \le
 C_{\mathrm{nl}}
 A_\phi(t)\|J(t)\phi(t)\|_{L^r}
\]
for almost every \(t\in[T,S]\). Having in mind \eqref{eq:Hpair}, H\"older's inequality provides
\begin{align}
 \|J\mathcal N_{0,\mu}(\phi)\|_
 {L^{a'}_tL^{r'}_x([T,S])}
 &\le
 C_{\mathrm{nl}}
 \|A_\phi\|_{L^\kappa((T,S))}
 \|J\phi\|_{L^a_tL^r_x([T,S])}.
 \label{eq:BJTH}
\end{align}
Combining \eqref{eq:BJTS2} and \eqref{eq:BJTH}, we obtain
\begin{align}\label{eq:BJTU0}
 \|J\phi\|_{L^a_tL^r_x([T,S])}
 &\le
 C_{\mathrm{Str}}
 \|J(T)\phi(T)\|_{L^2}
 +
 C_*
 \|A_\phi\|_{L^\kappa((T,S))}
 \|J\phi\|_{L^a_tL^r_x([T,S])}.
\end{align}
Because \(T\ge T_0\), we notice that
\[
 \|A_\phi\|_{L^\kappa((T,S))}
 \le
 \|A_\phi\|_{L^\kappa((T_0,\infty))}.
\]
Thus \eqref{eq:BJTS1} allows us to absorb the last term on the right hand side of \eqref{eq:BJTU0}, and we acquire
\begin{equation}\label{eq:BJTU}
 \|J\phi\|_{L^a_tL^r_x([T,S])}
 \le
 2C_{\mathrm{Str}}
 \|J(T)\phi(T)\|_{L^2},
\end{equation}
with a constant independent of \(S\). Letting \(S\to+\infty\) in \eqref{eq:BJTU} we obtain
\(
 J\phi
 \in
 L^a\bigl((T,\infty);L^r\bigr).
\)
By making a use of \eqref{eq:BJNE} and H\"older's inequality, we then find
\[
 \|J\mathcal N_{0,\mu}(\phi)\|_
 {L^{a'}((T,\infty);L^{r'})}
 \le
 C_{\mathrm{nl}}
 \|A_\phi\|_{L^\kappa((T,\infty))}
 \|J\phi\|_{L^a((T,\infty);L^r)}
 <\infty.
\]
It remains to prove the convergence of
\(xe^{-it\Delta}\phi(t)\) in \(L^2\). Let
\(
 T\le t_1<t_2.
\)
Lemma~\ref{lem:BJLS1}, applied on \([0,t_2]\), gives
\eqref{eq:BJCD} with \(s=t_1\) and \(t=t_2\), that is
\[
 J(t_2)\phi(t_2)
 =
 e^{i(t_2-t_1)\Delta}J(t_1)\phi(t_1)
 -
 i\int_{t_1}^{t_2}
 e^{i(t_2-s)\Delta}
 J(s)\mathcal N_{0,\mu}(\phi(s))\,ds.
\]
Applying \(e^{-it_2\Delta}\) to the above identity, we obtain
\[
 e^{-it_2\Delta}J(t_2)\phi(t_2)
 -
 e^{-it_1\Delta}J(t_1)\phi(t_1)
 =
 -i\int_{t_1}^{t_2}
 e^{-is\Delta}
 J(s)\mathcal N_{0,\mu}(\phi(s))\,ds.
\]
The dual Strichartz estimate \eqref{eq:BJDS} therefore gives
\begin{align}
 \left\|
  e^{-it_2\Delta}J(t_2)\phi(t_2)
  -
  e^{-it_1\Delta}J(t_1)\phi(t_1)
 \right\|_{L^2}
 \notag\\
 \lesssim
 \|J\mathcal N_{0,\mu}(\phi)\|_
 {L^{a'}((t_1,t_2);L^{r'})}
 \le
 \|J\mathcal N_{0,\mu}(\phi)\|_
 {L^{a'}((t_1,\infty);L^{r'})}.
 \label{eq:BJICB}
\end{align}
As
\(
 J\mathcal N_{0,\mu}(\phi)
 \in L^{a'}\bigl((T,\infty);L^{r'}\bigr),
\)
the right-hand side of \eqref{eq:BJICB} tends to zero as
\(t_1\to+\infty\), uniformly for \(t_2\ge t_1\). Using
\eqref{eq:BJII}, we have
\[
 \left\|
 xe^{-it_2\Delta}\phi(t_2)
 -
 xe^{-it_1\Delta}\phi(t_1)
 \right\|_{L^2(\R^d;\C^d)}
 \longrightarrow0
\]
as \(t_1\to+\infty\), uniformly for \(t_2\ge t_1\).
By \eqref{eq:BJII} and the completeness of
\(L^2(\R^d;\C^d)\), there exists
\(\eta_+\in L^2(\R^d;\C^d)\) such that
\(
 xe^{-it\Delta}\phi(t)
 \longrightarrow
 \eta_+
\)
in \(L^2(\R^d;\C^d)\) as \(t_1\to+\infty.\) This proves \eqref{eq:BJPML}.
\end{proof}
We finally have the following:
\begin{proof}[Proof of Theorem~\ref{thm:IMSB}]
Since
\(
 q>q_0>q_{\mathrm sr}
\)
with \(q_{\mathrm sr}\) as in \eqref{eq:mass1} and, by \eqref{eq:frf2},
\(
 q<q_{\mathrm m},
\)
with  \(q_{\mathrm m}\) as in \eqref{eq:BMa}, the assumptions \eqref{eq:ass} and \eqref{eq:frf} of
Theorem~\ref{thm:MH1} are satisfied with \(b=0\). Hence there exist
\(\phi_\pm\in H^1(\mathbb R^d)\) such that
\begin{equation}\label{eq:scattfin}
 \lim_{t\to\pm\infty}
 \left\|
  e^{-it\Delta}\phi(t)-\phi_\pm
 \right\|_{H^1}
 =0.
\end{equation}
We first consider \(t\to+\infty\). By
Proposition~\ref{prop:BJDT}, there exists
\(
 \eta_+\in L^2(\mathbb R^d;\mathbb C^d)
\)
such that
\begin{equation}\label{eq:IMMP}
 \lim_{t\to+\infty}
 \left\|
  xe^{-it\Delta}\phi(t)-\eta_+
 \right\|_{L^2(\mathbb R^d;\mathbb C^d)}
 =0.
\end{equation}
Set
\(
 v(t):=e^{-it\Delta}\phi(t).
\)
By \eqref{eq:BJII}, one obtains
\[
 xv(t)=e^{-it\Delta}J(t)\phi(t)
 \in L^2(\mathbb R^d;\mathbb C^d),
\]
so \(v(t)\) belongs to the domain 
\[
 X:D(X)\subset L^2(\mathbb R^d)
 \longrightarrow L^2(\mathbb R^d;\mathbb C^d),
 \]
of the multiplication operator \(
 Xf:=xf,
 \)
where
\[
 D(X):=
 \left\{
  f\in L^2(\mathbb R^d):
  xf\in L^2(\mathbb R^d;\mathbb C^d)
 \right\}.
\]
It is straightforward to see that the operator \(X\) is closed. Let \(t_n\to+\infty\). From \eqref{eq:scattfin} and \eqref{eq:IMMP}, we have
\(
 v(t_n)\longrightarrow\phi_+
 \)
 in
 \(
 L^2,
\)
and
\(
 Xv(t_n)\longrightarrow\eta_+
 \)
 in
 \(
 L^2(\mathbb R^d;\mathbb C^d).
\)
This fact then provides
\(
 \phi_+\in D(X),
 \,
  X\phi_+=\eta_+.
\)
Thus
\(
 x\phi_+=\eta_+
 \in L^2(\mathbb R^d;\mathbb C^d),
\)
and, since \(\phi_+\in H^1(\mathbb R^d)\), we conclude that
\(
 \phi_+\in\Sigma.
\)
Moreover, \eqref{eq:IMMP} becomes
\[
 \lim_{t\to+\infty}
 \left\|
  xe^{-it\Delta}\phi(t)-x\phi_+
 \right\|_{L^2}
 =0.
\]
Combining this with the preceding \(H^1\)-convergence yields
\[
 \lim_{t\to+\infty}
 \left\|
  e^{-it\Delta}\phi(t)-\phi_+
 \right\|_{\Sigma}
 =0.
\]
For negative times, consider
\(
 \widetilde\phi(t,x):=\overline{\phi(-t,x)}.
\)
The preceding argument then bears
\(\overline{\phi_-}\in\Sigma\) and, in view of the fact that the complex conjugation preserves
the \(\Sigma\)-norm, we get
\[
 \lim_{t\to-\infty}
 \left\|
  e^{-it\Delta}\phi(t)-\phi_-
 \right\|_{\Sigma}
 =0.
\]
Together with the positive-time conclusion, this shows
\eqref{eq:IMSC}.
\end{proof}

\appendix

\section{Control of the second moment}
\label{sec:BLWS}

We use the notation of
Section~\ref{sec:BJUS}. In particular,
\(\mathcal N_{0,\mu}\), \(J(t)\), and the admissible pair \((a,r)\) are
defined in that section, see \eqref{eq.GVect} and
\eqref{eq:BJAP}. The conformal space
\(\Sigma\) is defined in \eqref{eq:introS}. We have

\begin{proposition}
\label{prop:FMP}
Assume \eqref{eq:ass}. Let \(I\subset\mathbb R\) be an interval with
\(0\in I\), and let
\(
\phi\in C\bigl(I;H^1(\mathbb R^d)\bigr)
\)
be an \(H^1\)-solution of \eqref{eq:ME} with initial data 
\(\phi_0\in\Sigma\). Then
\(
 x\phi\in C\bigl(I;L^2(\mathbb R^d)\bigr),
 \)
 and hence
\(
 \phi\in C(I;\Sigma).
\)
Moreover, for every \(s,t\in I\),
\begin{equation}\label{eq:MPG}
 \left|
  \|x\phi(t)\|_{L^2}
  -
  \|x\phi(s)\|_{L^2}
 \right|
 \le
 2\int_{\min\{s,t\}}^{\max\{s,t\}}
  \|\nabla\phi(\tau)\|_{L^2}\,d\tau.
\end{equation}
In particular, for \(s=0\) one gets
\eqref{eq:IMP}.
\end{proposition}

\begin{proof}
For \(R>0\), one sets
\[
 \rho_R(x)
 :=
 \frac{|x|^2}{1+|x|^2/R^2},
 \qquad
 M_R(t)
 :=
 \int_{\R^d}\rho_R(x)|\phi(t,x)|^2\,dx.
\]
Then, for every \(x\in\R^d\), we note that
\[
 0\le \rho_R(x)\le |x|^2,
 \qquad
 \lim_{R\to\infty}\rho_R(x)=|x|^2,
\]
with \(\rho_R(x)\) a nondecreasing function with respect to \(R\). Moreover, we have
\[
 |\nabla\rho_R(x)|^2
 \le
 4\rho_R(x).
\]
The standard virial identity shows that
\[
 M_R'(t)
 =
 2\operatorname{Im}
 \int_{\R^d}
 \nabla\rho_R\cdot
 \overline{\phi(t)}\,\nabla\phi(t)\,dx
\]
for almost every \(t\in I\). Hence
\[
 |M_R'(t)|
 \le
 4M_R(t)^{1/2}\|\nabla\phi(t)\|_{L^2}.
\]
If \(\phi_0=0\), then conservation of mass \eqref{eq:CM2} implies
\(\phi\equiv0\), and the conclusion is immediate. We may therefore
assume that \(\phi_0\neq0\). Again, by \eqref{eq:CM2},
\[
 \|\phi(t)\|_{L^2}
 =
 \|\phi_0\|_{L^2}>0,
 \qquad
 t\in I.
\]
Since \(\rho_R>0\) almost everywhere in \(\R^d\), it follows that
\(
 M_R(t)>0,
\)
for every
\(
 t\in I.
\)
Thus, on every compact subinterval of \(I\), the continuous function
\(M_R\) is bounded away from zero. The chain rule therefore gives
\[
 \left|
 \frac{d}{dt}M_R(t)^{1/2}
 \right|
 =
 \frac{|M_R'(t)|}{2M_R(t)^{1/2}}
 \le
 2\|\nabla\phi(t)\|_{L^2}
\]
for almost every \(t\in I\). Integrating between \(s,t\in I\), we
obtain
\begin{equation}\label{eq:FMP1}
 \left|
 M_R(t)^{1/2}-M_R(s)^{1/2}
 \right|
 \le
 2\int_{\min\{s,t\}}^{\max\{s,t\}}
 \|\nabla\phi(\tau)\|_{L^2}\,d\tau.
\end{equation}
Taking \(s=0\) and using \(\rho_R\le|x|^2\), we find
\[
 M_R(t)^{1/2}
 \le
 \|x\phi_0\|_{L^2}
 +
 2\int_{\min\{0,t\}}^{\max\{0,t\}}
 \|\nabla\phi(\tau)\|_{L^2}\,d\tau.
\]
The right-hand side of the above inequality is finite. Additionally, monotone convergence gives
\(
 x\phi(t)\in L^2(\R^d),
 \)
 and
 \[
 \lim_{R\to\infty}M_R(t)^{1/2}
 =
 \|x\phi(t)\|_{L^2}.
\]
Letting \(R\to\infty\) in \eqref{eq:FMP1} yields
\[
 \left|
 \|x\phi(t)\|_{L^2}
 -
 \|x\phi(s)\|_{L^2}
 \right|
 \le
 2\int_{\min\{s,t\}}^{\max\{s,t\}}
 \|\nabla\phi(\tau)\|_{L^2}\,d\tau,
\]
which proves \eqref{eq:MPG}. Finally, let
\(
G\in C_c^\infty(\R^d;\C^d),
\)
then
\[
 \left\langle x\phi(t),G\right\rangle_{L^2(\R^d;\C^d)}
 =
 \left\langle \phi(t),x\cdot G\right\rangle_{L^2(\R^d)}.
\]
Since \(x\cdot G\in L^2(\R^d)\) and
\(\phi\in C(I;L^2(\R^d))\), the function on the right-hand side of the above identity is continuous in
\(t\). By density and the local boundedness ensured
by \eqref{eq:MPG}, it follows that
\(t\mapsto x\phi(t)\) is weakly continuous in \(L^2\).
Together with the fact that 
\(\|x\phi(t)\|_{L^2}\) is continuous in \(t\), this bears
\(
 x\phi\in C\bigl(I;L^2(\R^d)\bigr),
\)
and hence \(\phi\in C(I;\Sigma)\).

\end{proof}
\section{Bilinear estimates for the Riesz kernel}
\label{sec:RIS}

We begin with an auxiliary lemma for the Riesz kernel used in the
scattering arguments. That is,

\begin{lemma}
\label{lem:BRFI}
Let
\(
 0<\mu<d,
 \)
 \(
 I_\mu(x)
 \)
 defined as in \eqref{eq:Nonl} and
 \begin{equation}\label{eq:mu}
  p_\mu:=\frac{2d} {d+\mu}.
\end{equation}
For nonnegative functions \(F,G\in L^{p_\mu}(\R^d)\), define
\begin{equation}
 B(F,G)
 :=
 \iint_{\R^d\times\R^d}
 I_\mu(x-y)F(x)G(y)\,dx\,dy.
\end{equation}
Then \(B(F,G)\) is finite and symmetric, and
\begin{equation}\label{eq:BHLS}
 B(F,G)
 \le
 C_{d,\mu}
 \|F\|_{L^{p_\mu}}
 \|G\|_{L^{p_\mu}}.
\end{equation}
Moreover, the following inequalities hold.

\begin{enumerate}
\item[\rm(i)]
Let \(N\in\N\), let
\[
 F_j,G_j\in L^{p_\mu}(\R^d),
 \qquad
 F_j,G_j\ge0,
\]
and let \(\theta_j\ge0\) satisfy
\[
 \sum_{j=1}^N\theta_j=1.
\]
Writing
\[
 \mathcal J
 :=
 \{j\in\{1,\dots,N\}:\theta_j>0\},
\]
one has
\begin{equation}\label{eq:BBH}
 B\left(
  \prod_{j\in\mathcal J}F_j^{\theta_j},
  \prod_{j\in\mathcal J}G_j^{\theta_j}
 \right)
 \le
 \prod_{j\in\mathcal J}
 B(F_j,G_j)^{\theta_j}.
\end{equation}

\item[\rm(ii)]
For all nonnegative \(F,G\in L^{p_\mu}(\R^d)\),
\begin{equation}\label{eq:BBC}
 B(F,G)^2
 \le
 B(F,F)B(G,G).
\end{equation}
\end{enumerate}
\end{lemma}

\begin{proof}
Because \(0<\mu<d\), one has
\(
 1<p_\mu<2,
\)
and also
\[
 \frac{2}{p_\mu}+\frac{d-\mu}{d}=2.
\]
The Hardy-Littlewood-Sobolev inequality therefore leads to
\begin{equation}
 \iint_{\R^d\times\R^d}
 \frac{F(x)G(y)}{|x-y|^{d-\mu}}
 \,dx\,dy
 \le
 C_{d,\mu}
 \|F\|_{L^{p_\mu}}
 \|G\|_{L^{p_\mu}}.
\end{equation}
For this reason \(B(F,G)\) is finite and
\eqref{eq:BHLS} holds. Symmetry follows from
\(
 I_\mu(x-y)=I_\mu(y-x).
\)
We first prove \eqref{eq:BBH}. Define the positive
measure
\[
 d\nu(x,y)
 :=
 I_\mu(x-y)\,dx\,dy
\]
on \(\R^d\times\R^d\). This measure is \(\sigma\)-finite. in fact, if
\(B_n:=\{x\in\R^d:|x|\le n\}\), then
\begin{align*}
 \nu(B_n\times B_n)
 &=
 \int_{B_n}\int_{B_n}
 |x-y|^{-(d-\mu)}
 \,dy\,dx
\le
 |B_n|
 \int_{B_{2n}}
 |z|^{-(d-\mu)}
 \,dz
 <\infty.
\end{align*}
For \(j\in\mathcal J\), set
\(
 s_j:=1/\theta_j,
\)
so we achieve
\[
 1\le s_j<\infty,
 \qquad
 \sum_{j\in\mathcal J}\frac1{s_j}
 =
 \sum_{j\in\mathcal J}\theta_j
 =1.
\]
Generalized H\"older's inequality on
\((\R^d\times\R^d,\nu)\) gives
\begin{align*}
 &B\left(
  \prod_{j\in\mathcal J}F_j^{\theta_j},
  \prod_{j\in\mathcal J}G_j^{\theta_j}
 \right)
=
 \iint_{\R^d\times\R^d}
 \prod_{j\in\mathcal J}
 \bigl(F_j(x)G_j(y)\bigr)^{\theta_j}
 \,d\nu(x,y)
 \\
 &\quad\le
 \prod_{j\in\mathcal J}
 \left(
  \iint_{\R^d\times\R^d}
  F_j(x)G_j(y)\,d\nu(x,y)
 \right)^{\theta_j}
 =
 \prod_{j\in\mathcal J}
 B(F_j,G_j)^{\theta_j},
\end{align*}
and this shows \eqref{eq:BBH}. The indices satisfying
\(\theta_j=0\) are omitted to avoid any ambiguity involving zero
powers. We next prove \eqref{eq:BBC}. Suppose first that
\(
 F,G\in C_c^\infty(\R^d),
 \,
 F,G\ge0.
\)
Their Fourier transforms satisfy
\[
 \widehat F(-\xi)=\overline{\widehat F(\xi)},
 \qquad
 \widehat G(-\xi)=\overline{\widehat G(\xi)}.
\]
Also we recall that
\begin{equation}\label{eq:BRFT}
 \widehat{I_\mu}(\xi)
 =
 c_{d,\mu}|\xi|^{-\mu},
 \qquad
 c_{d,\mu}>0,
\end{equation}
Because \(\widehat F\) and \(\widehat G\) are Schwartz functions, we have
\(
 |\xi|^{-\mu}
 |\widehat F(\xi)|
 |\widehat G(\xi)|
 \in L^1(\R^d).
\)
In fact, \(|\xi|^{-\mu}\) is locally integrable as \(\mu<d\),
while the rapid decay of the Fourier transforms gives integrability
at infinity. Using \eqref{eq:BRFT} and Fourier duality, we
obtain
\begin{equation}
 B(F,G)
 =
 c_{d,\mu}
 \int_{\R^d}
 |\xi|^{-\mu}
 \widehat F(\xi)
 \overline{\widehat G(\xi)}
 \,d\xi.
\end{equation}
In particular,
\begin{equation}
 B(F,F)
 =
 c_{d,\mu}
 \int_{\R^d}
 |\xi|^{-\mu}
 |\widehat F(\xi)|^2\,d\xi
 \ge0.
\end{equation}
Applying the Cauchy--Schwarz inequality with respect to the measure
\(
 |\xi|^{-\mu}\,d\xi
\)
yields
\begin{align*}
 B(F,G)^2
 &\le
 c_{d,\mu}^2
 \left(
  \int_{\R^d}
  |\xi|^{-\mu}
  |\widehat F(\xi)|^2\,d\xi
 \right)
 \left(
  \int_{\R^d}
  |\xi|^{-\mu}
  |\widehat G(\xi)|^2\,d\xi
 \right)
 =
 B(F,F)B(G,G).
\end{align*}
One can pass to arbitrary nonnegative
\(F,G\in L^{p_\mu}(\R^d)\) by a standard density argument.

\end{proof}

\begin{remark}
The restriction to nonnegative functions in
Lemma~\ref{lem:BRFI} is adapted to nonnegative functions, since the arguments of \(B\) are powers of
absolute values. More generally, the same formula defines a continuous
symmetric bilinear form on the real space
\(L^{p_\mu}(\R^d)\). It is positive semidefinite and satisfies
\[
 |B(F,G)|^2
 \le
 B(F,F)B(G,G),
 \qquad
 B(F,F)\ge0
\]
for arbitrary real-valued
\(F,G\in L^{p_\mu}(\R^d)\). These properties follow from the
non-negativity of the Fourier multiplier
\(c_{d,\mu}|\xi|^{-\mu}\).
\end{remark}

\end{document}